\documentclass[a4paper]{amsart}
\usepackage{amssymb}
\usepackage[polutonikogreek, english]{babel}
\usepackage{graphicx}
\usepackage{pstricks}
\usepackage{enumerate}
\theoremstyle{plain}
\newtheorem{theorem}{Theorem}
\newtheorem{corollary}[theorem]{Corollary}
\newtheorem{lemma}[theorem]{Lemma}
\newtheorem{axiom}{Axiom}
\newtheorem{definition}{Definition}
\author{Wim Veldman}
\address{Institute for Mathematics, Astrophysics and Particle Physics, Faculty of Science, Radboud University Nijmegen,
Postbus 9010, 6500 GL Nijmegen, the Netherlands}
\email{W.Veldman@science.ru.nl}
\title{Equality and equivalence, intuitionistically}
\dedicatory{For Mohammad Ardeshir, in friendship\\Solem enim e mundo tollere videntur qui amicitiam e vita tollunt.\\They take away the sun from the world, surely, those who take away friendship from life.\\\flushright Cicero, \emph{de Amicitia, XIII 47}}

\begin{document}

 \begin{abstract}We show that the intuitionistic first-order theory of equality has continuum many complete extensions. We also study the Vitali equivalence relation and show there are many intuitionistically precise versions of it. \end{abstract}
 
 \maketitle
 \section{Introduction}
 
 We want to contribute to L.E.J.~Brouwer's program of doing mathematics \\ \textit{intuitionistically}. 
 
 \medskip
 We follow his advice to interpret the logical constants constructively. 
 
 A conjunction $A\;\wedge \;B$ is considered proven if and only if one has a proof of $A$ and also a proof of $B$.
 
 A disjunction $A\vee B$ is considered proven if and only if either $A$ or $B$ is proven.
 
 An implication $A\rightarrow B$ is considered proven if and only if there is a proof of $B$ using the assumption $A$. 
 
 A negation $\neg A$ is considered proven if and only if there is a proof of $A\rightarrow 0=1$.
 
 An existential statement $\exists x \in V[P(x)]$ is considered proven if and only an element $x_0$ is produced together with a proof of the associated statement $P(x_0)$. 
 
 A universal statement $\forall x\in V[P(x)]$ is considered proven if and only if a method is given that produces, given any $x$ in $V$, a proof of the associated statement $P(x)$.

 \medskip
 We also use some axioms  proposed by Brouwer: his \textit{Continuity Principle}, our Axiom \ref{ax:bcp}, a slightly stronger version of it, the \textit{First Axiom of Continuous Choice}, our Axiom \ref{ax:fcc}, and his \textit{Thesis on Bars in $\mathcal{N}$}, our Axiom \ref{A:barinduction}. 
 
 In some of our proofs, we use an  Axiom of Countable Choice, our Axiom \ref{ax:countable choice}.  Intuitionistic mathematicians, who accept infinite step-by-step constructions not determined by a rule, consider  this axiom a reasonable proposal.
 
 Finally, we believe that generalized inductive definitions, like our Definition \ref{D:perhapsiveextvitali}, fall within the compass of intuitionistic mathematics. 
 
 \medskip
Our subject is the (intuitionistic) first-order theory of equality.  By considering structures $(\mathcal{X}, =)$ where $\mathcal{X}$ is a subset of Baire space $\mathcal{N}=\omega^\omega$ and $=$ the usual equality relation on $\mathcal{N}$, we find that the theory has an uncountable and therefore astonishing\footnote{Classically, all infinite models of the first-order theory of equality are elementarily equivalent.} variety of elementarily different infinite models and, as a consequence, an astonishing variety of complete extensions, see Theorem \ref{T:uncountablymany}. The key observation\footnote{This observation has been made earlier in \cite[Section 5]{veldman2001}. The first part of the present paper elaborates part of \cite[Section 5]{veldman2001}.} leading to this result is the recognition that, in a \textit{spread}\footnote{Every  \textit{spread} is a closed subset of $\mathcal{N}$, see Section \ref{S:spreads}.}, an \textit{isolated} point is  the same as a \textit{decidable} point.\footnote{See Lemma \ref{L:defisol}. $\alpha\in \mathcal{X}\subseteq\mathcal{N}$ is a \textit{decidable} point of $\mathcal{X}$ if and only if $\forall \beta \in \mathcal{X}[\alpha=\beta\;\vee\;\neg(\alpha=\beta)]$.} It follows that the set of the non-isolated points of a spread is a definable subset of the  spread. In spreads that are \textit{transparent}\footnote{see Definition \ref{D:transparent}.}, the set of the non-isolated points of the spread coincides with the \textit{coherence} of the spread\footnote{The \textit{coherence} of a closed set is the set of its limit points, see Definition \ref{D:limitpoint}.},  and the coherence itself is spread.  It  may happen that the coherence of a transparent spread is transparent itself and then the  coherence of the coherence also is a definable subset of the spread.  And so on. 

\smallskip Any structure $(\mathcal{N}, R)$, where $R$ is an equivalence relation on $\mathcal{N}$, is a model of the theory of equality. We study the \textit{Vitali equivalence relation}, see Section \ref{S:vitali}, as an example.   This equivalence relation, in contrast to the equality relation on $\mathcal{N}$, is not \textit{stable}\footnote{$R\subseteq \mathcal{N}\times\mathcal{N}$ is called \textit{stable} if $\forall \alpha\forall\beta[\neg\neg\alpha R\beta \rightarrow \alpha R\beta]$, see Definition \ref{D:stable}.}, see Theorem \ref{T:vitaliunstable}. 

There is a host of binary relations on $\mathcal{N}$ that, from a classical point of view, all would be the same as the Vitali equivalence relation, see Sections \ref{S:vitalivar} and \ref{S:morevitali}, and especially Definition \ref{D:perhapsiveextvitali},  Corollary \ref{C:hierarchy} and Definition \ref{D:almost}.   It turned out to be difficult to find differences between them that are first-order expressible. We did find some such differences, however, by studying structures $(\mathcal{N}, =, R)$, where $R$ is an intuitionistic version of the  Vitali equivalence relation and $=$ the usual equality, see Section \ref{S: equequiv}. 

\medskip The paper is divided into 13 Sections and consists roughly of two parts. Sections 2-8 lead up to the result that the theory of equality has continuum many complete extensions, see Theorem  \ref{T:uncountablymany}. Sections 9-12 treat the Vitali equivalence relations. Section 13 lists some notations and conventions and may be used by the reader as a reference.  
 \section{Intuitionistic model theory} 
 
  Given a  relational structure $\mathfrak{A}=(A, R_0, R_1, \ldots, R_{n-1})$, we  construct a first-order language $\mathcal{L}$ with basic formulas $\mathsf{R}_i(\mathsf{x}_0, \mathsf{x}_1,\ldots, \mathsf{x}_{l_i-1})$, where $i<n$ and $l_i$ is the arity of $R_i$. The formulas of $\mathcal{L}$ are  obtained from the basic formulas by using $\wedge, \vee, \rightarrow, \neg, \exists, \forall$ in the usual way. 
 
 For every formula $\varphi=\varphi(\mathsf{x}_0, \mathsf{x}_1, \ldots, \mathsf{x}_{m-1})$ of $\mathcal{L}$, for all $a_0, a_1, \ldots, a_{m-1}$ in $A$, we define the statement:
$$\mathfrak{A}\models \varphi[a_0, a_1, \ldots, a_{m-1}]$$
 (\textit{$\mathfrak{A}$ realizes $\varphi$ if $\mathsf{x}_0, \mathsf{x}_1, \ldots, \mathsf{x}_{m-1}$ are interpreted by $a_0, a_1, \ldots, a_{m-1}$, respectively}), as Tarski did it, with the proviso that connectives and quantifiers are interpreted intuitionistically.
 
  A formula $\varphi$ of $\mathcal{L}$ without free variables will be called a \textit{sentence}.
  
  A \textit{theory (in $\mathcal{L}$)} is a set of sentences of $\mathcal{L}$. 
  
  Given  a theory $\Gamma$ in $\mathcal{L}$ and a structure $\mathfrak{A}$, we define: $\mathfrak{A}$ \textit{realizes} $\Gamma$ if and only if, for every $\varphi$ in $\Gamma$, $\mathfrak{A}\models \varphi$.
  
  Given a structure $\mathfrak{B}$ that has the same signature as $\mathfrak{A}$, so that the formulas of $\mathcal{L}$ may be interpreted in $\mathfrak{B}$ as well as in $\mathfrak{A}$, we let  $Th(\mathfrak{B})$, the \textit{theory of $\mathfrak{B}$},  be the set of all sentences $\varphi$ of $\mathcal{L}$ such that $\mathfrak{B}\models \varphi$. 
  
  A theory $\Gamma$ in $\mathcal{L}$ will be called a \textit{complete theory} if and only if there exists a structure $\mathfrak{B}$ such that $\Gamma=Th(\mathfrak{B})$. 
  
  This agrees with one of the uses of the expression `\textit{complete theory}' in classical, that is: usual, non-intuitionistic, model theory, see \cite[p. 43]{hodges}. Note that one may be unable to decide, for a given sentence $\varphi$ and a given structure $\mathfrak{B}$, whether or not $\mathfrak{B}\models \varphi$. Intuitionistically,  it is not true that, for every complete theory $\Gamma$ and every sentence $\varphi$,  \textit{either} $\varphi\in \Gamma$ \textit{or} $\neg\varphi\in\Gamma$.

 \smallskip Complete theories $\Gamma, \Delta$ are \textit{positively different} if one may point out a sentence $\psi$ such that $\psi \in \Gamma$ and 
 $\neg\psi \in \Delta$.\footnote{
 If $\psi \in \Gamma$ and 
 $\neg\psi \in \Delta$, then $\neg \psi \in \Delta$ and $\neg\neg\psi \in \Gamma$: the relation \textit{positively different} is symmetric.}
 
 \smallskip Structures $\mathfrak{A}, \mathfrak{B}$ are \textit{elementarily equivalent} if and only if $Th(\mathfrak{A})=Th(\mathfrak{B})$ and \textit{(positively) elementarily different} if $Th(\mathfrak{A})$ is positively different from $ Th(\mathfrak{B})$.

 \smallskip
 Let $\Gamma$ be a theory in $\mathcal{L}$. A good question is the following:
 
 \begin{quote} \textit{How many complete theories $\Delta$ can one find extending $\Gamma$?}\end{quote}
 
 We will say: \textit{$\Gamma$ admits countably many complete extensions} if and only if there exists an infinite sequence $\Delta_0, \Delta_1, \ldots$ of complete theories extending $\Gamma$ such that, for all $m,n$, if $m\neq n$, then $\Delta_m,\Delta_n$ are (positively) different, and
 
 \textit{$\Gamma$ admits continuum many complete extensions} if and only if there exists a function $\alpha\mapsto \Delta_\alpha$ associating to every element $\alpha$ of $\mathcal{C}=2^\omega$ a complete theory extending $\Gamma$ such that  for all $\alpha, \beta$, if\footnote{$\alpha\;\#\;\beta\leftrightarrow\alpha\perp\beta\leftrightarrow \exists n[\alpha(n)\neq\beta(n)]$, see Section \ref{S:notations}.} $\alpha\;\#\;\beta$, then $\Delta_\alpha,\Delta_\beta$ are (positively) different.
 
 \smallskip
 A main result of this paper is that the first-order theory of equality admits continuum many complete extensions.

\section{Equality may be undecidable}The first-order theory $EQ$ of equality consists of the following three axioms: \begin{enumerate} \item $\mathsf{\forall x[x=x]}$, \item $\mathsf{\forall x\forall y[x=y\rightarrow y=x]}$ and \item $\mathsf{\forall x\forall y\forall z[(x=y\;\wedge\;y=z)\rightarrow x=z]}$.
\end{enumerate}

A model of $EQ$ is a structure of the form $(V,R)$, where $V$ is a set and $R$ is an equivalence relation on $V$, possibly, but not necessarily, the  equality relation belonging to $V$.  

\medskip

Classically, every  complete extension of $EQ$ is realized in one of the structures from the list: $(\{0\},=)$, $(\{0,1\},=)$, $(\{0,1,2\},=)$, $\dots$ and   $(\omega,=)$. This shows that, classically, $EQ$ admits of (no more than) countably many complete extensions.

\smallskip

Intuitionistically, however, we have to observe that all structures on this  list satisfy the sentence

\begin{quote}$ \mathsf{\forall x\forall y[x=y\;\vee\;\neg(x=y)]}$, \end{quote}
that is: the equality relation, on each of these sets, is a \textit{decidable} relation. 

\medskip
It is well-known, however,  that the equality relation on the set $\mathcal{N}$ is \textit{not} a decidable relation. Let us recall why.

If we define an element $\alpha$ of $\mathcal{N}$ by stipulating: $$\forall n[\alpha(n)\neq 0\leftrightarrow \forall i<99[d(n+i)=9]],$$ where  $d:\mathbb{N}\rightarrow \{0,1,\ldots,9\}$ is the decimal expansion of $\pi$, we are unable to decide: $\alpha=\underline 0\;\vee\neg(\alpha =\underline 0)$. 

This is because, if $\alpha =\underline 0$, then $\neg\exists n\forall i<99[\alpha(n+i)=9]$, and, if $\neg(\alpha=\underline 0)$, then \\$\neg\neg\exists n\forall i<99[d(n+i)=9]$, and  we have no proof of either alternative. 

This example shows us that the statement $\forall \alpha[\alpha =\underline 0\;\vee\;\neg(\alpha =\underline 0)]$, for a constructive mathematician, who interprets the disjunction strongly, is  a \textit{reckless} statement.\footnote{A statement is \textit{reckless} if one might think it is true while the intuitionistic mathematician understands there is no proof for his constructive reading of it.}  

The following axiom, used by Brouwer\footnote{see \cite{veldman2001}.}, implies that the statement  \\$\forall \alpha[\alpha =\underline 0\;\vee\;\neg(\alpha =\underline 0)]$ even leads to a contradiction.

\begin{axiom}[Brouwer's Continuity Principle]\label{ax:bcp} $\;\;$\\ For all $R\subseteq \mathcal{N}\times\omega$, if $\forall \alpha\exists n[\alpha Rn]$, then $\forall \alpha\exists m \exists n\forall \beta[\overline\alpha m \sqsubset\beta\rightarrow \beta Rn]$. \end{axiom}

An immediate consequence is:\begin{lemma}[Brouwer's Continuity Principle, the case of disjunction]\label{L:bcpdisj} $\;\;$ \\For all $P_0, P_1\subseteq \mathcal{N}$, if $\forall \alpha [\alpha\in P_0\;\vee\;\alpha \in P_1]$, then \\$\forall \alpha\exists m[\forall \beta[\overline \alpha m\sqsubset \beta\rightarrow \beta \in P_0]\;\vee\; \forall \beta[\overline \alpha m\sqsubset \beta\rightarrow \beta \in P_1]]$. \end{lemma}

\begin{proof} Define $R:=\{(\alpha, n)\mid \alpha \in P_n]$ and apply Axiom \ref{ax:bcp}. \end{proof}

\begin{theorem} \begin{enumerate}[\upshape (i)] \item $(\mathcal{N}, =)\models \mathsf{\forall x\neg\forall y[x=y\;\vee\;\neg(x=y)]}$. \item $(\mathcal{N}, =)\models \mathsf{\neg\forall x\forall y[x=y\;\vee\;\neg(x=y)]}$. \end{enumerate} \end{theorem}

\begin{proof} (i) Let $\alpha$ be given and assume: $\forall \beta[\alpha=\beta\;\vee\;\neg(\alpha=\beta)]$.\\ Using Lemma \ref{L:bcpdisj}, find $m$ such that \\\textit{either} $\forall \beta[\overline \alpha m\sqsubset\beta \rightarrow \alpha = \beta]$ \textit{or} $\forall \beta[\overline \alpha m\sqsubset\beta \rightarrow \neg(\alpha = \beta)]$. 
\\Consider $\beta:=\overline\alpha m\ast\langle \alpha(m)+1\rangle\ast\underline 0$ (for the first alternative) and $\beta:=\alpha$ (for the second one)  and conclude that both alternatives are false.

\smallskip
 
(ii) This is an immediate consequence of (i). \end{proof}

\begin{definition}\label{D:Tinf}$\;$

 For each $n$,  we let  $\psi_n$ be the sentence $\mathsf{\exists x_0\exists x_1\ldots \exists x_n}[\bigwedge_{i<j<n}\neg (\mathsf{x}_i=\mathsf{x}_j)]$. 

 $T_{inf}:=EQ\cup\{\psi_n\mid n\in \omega\}$.\end{definition}
$\psi_n$ expresses that a set has at least $n+1$ elements.

Note that, in classical mathematics, $T_{inf}$ has only one complete extension.  

Intuitionistically, however, $T_{inf}$ has  (at least) two positively different complete extensions, $Th\bigl((\mathcal{N},=)\bigr)$ and $Th\bigl((\omega, =)\bigr)$.

The next Theorem reflects the fact that, in classical model theory, all  models of $T_{inf}$ are elementarily equivalent. 

\begin{theorem}\label{T:classicalcompleteness} The theory $T_{inf}\cup \{\forall\mathsf{x\forall y[x=y\;\vee\;\neg(x=y)]}\}$ has only one complete extension. \end{theorem}

\begin{proof} For each $n$,  consider the first $n$ variables of  our language: $\mathsf{x}_0, \mathsf{x}_1, \ldots ,\mathsf{x}_{n-1}$.  A formula $\varepsilon=\varepsilon(\mathsf{x}_0, \mathsf{x}_1, \ldots ,\mathsf{x}_{n-1})$ is called an \textit{equality type} if and only if it is of the form $\bigwedge_{i<j<n} \sigma_{ij}$ where each $\sigma_{ij}$ either is the formula $\mathsf{x}_i=\mathsf{x}_j$ or the formula $\neg(\mathsf{x}_i=\mathsf{x}_j)$.\footnote{\textit{Inconsistent}  equality types may be annoying but do not cause difficulties.}
One may prove: for all structures $(V_0, R_0), (V_1, R_1)$, both realizing $T_{inf}\cup \{\forall\mathsf{x\forall y[x=y\;\vee\;\neg(x=y)]}\}$, for each formula $\varphi=\varphi(\mathsf{x}_0, \mathsf{x}_1, \ldots ,\mathsf{x}_{n-1})$, for each equality type $\varepsilon=\varepsilon(\mathsf{x}_0, \mathsf{x}_1, \ldots ,\mathsf{x}_{n-1})$, $(V_0, R_0)\models \forall \mathsf{x}_0\forall\mathsf{x}_1 \ldots\forall\mathsf{x}_{n-1}[\varepsilon \rightarrow \varphi]$ if and only if  $(V_1, R_1)\models \forall \mathsf{x}_0\forall\mathsf{x}_1 \ldots\forall\mathsf{x}_{n-1}[\varepsilon \rightarrow \varphi]$. The proof is by induction on the complexity of the formula $\varphi$. 

It follows that any two models $(V_0, R_0), (V_1, R_1)$, both realizing \\$T_{inf}\cup \{\forall\mathsf{x\forall y[x=y\;\vee\;\neg(x=y)]}\}$, are elementarily equivalent. \end{proof}

From here on, we restrict attention to infinite models of $EQ$, that is, to models of $T_{inf}$. The hackneyed question to make a survey of models that are \textit{finite}, or at least \textit{not infinite}, and of models for which one can not decide if they are finite or infinite, is left for another occasion. That the job is not an easy one will be clear to readers of \cite{veldman1995}. 
\section{Spreads}\label{S:spreads}

\begin{definition}Let $\beta$ be given. $\beta$ is called a \emph{spread-law}, $Spr(\beta)$, if and only if
$\forall s[\beta(s)=0\leftrightarrow \exists n[\beta(s\ast\langle n \rangle)=0]]$.

For every $\beta$, we define: $\mathcal{F}_\beta:=\{\alpha\mid\forall n[\beta(\overline \alpha n)=0]\}$. 

$\mathcal{X}\subseteq \mathcal{N}$ is  \emph{closed}  if and only if $\exists \beta[\mathcal{X}=\mathcal{F}_\beta]$. 

$\mathcal{X}\subseteq \mathcal{N}$ is a \emph{spread}  if and only if $\exists \beta[Spr(\beta)\;\wedge\;\mathcal{X}=\mathcal{F}_\beta]$. \end{definition}

If $Spr(\beta)$ and $\beta(\langle \;\rangle)\neq 0$, then $\mathcal{F}_\beta=\emptyset$.

 If $Spr(\beta)$ and $\beta(\langle\;\rangle)=0$, then $\mathcal{F}_\beta$ is inhabited\footnote{$\mathcal{X}\subseteq\mathcal{N}$ is \textit{inhabited} if and only if $\exists \alpha[\alpha\in \mathcal{X}]$.}.  One may define $\alpha$ such that \\$\forall n[\alpha(n)=\mu p[\beta(\overline\alpha n\ast \langle p \rangle)=0]]$ and observe: $\forall n[\beta(\overline \alpha n)=0]$, that is: $\alpha \in \mathcal{F}_\beta$.

\medskip
Is every closed set  a spread?

Define $\beta$ such that $\forall s[\beta(s) =0\leftrightarrow \neg\forall i<99[d(n+i)=9]],$ where \\$d:\mathbb{N}\rightarrow \{0,1,\ldots, 9\}$ is the decimal expansion of $\pi$. 

If $\mathcal{F}_\beta$ is a spread, that is $\exists\gamma[Spr(\gamma)\;\wedge\;\mathcal{F}_\gamma=\mathcal{F}_\beta]$,  then \textit{either} $\mathcal{F}_\beta$ is inhabited and $\neg \exists s\forall i<99[d(s+i)=9]$ \textit{or} $\mathcal{F}_\beta = \emptyset$ and $\neg \neg\exists s\forall i<99[d(s+i)=9]$.

 For this $\beta$,  the statement `$\mathcal{F}_\beta$ is a spread' thus turns out to be reckless.
 
 \smallskip
 Brouwer's Continuity Principle  enables one to obtain a stronger conclusion.

\begin{theorem} $\neg\forall \beta\exists \gamma[Spr(\gamma)\;\wedge\;\mathcal{F}_\gamma=\mathcal{F}_\beta]$. \end{theorem}

\begin{proof} Assume: $\forall \beta\exists \gamma[Spr(\gamma)\;\wedge\;\mathcal{F}_\gamma=\mathcal{F}_\beta]$.
Then $\forall \beta[\exists\alpha[\alpha \in \mathcal{F}_\beta]\;\vee\;\neg\exists\alpha[\alpha\in \mathcal{F}_\beta]]$.
Using Lemma \ref{L:bcpdisj}, find $m$ such that \textit{either} $\forall \beta[\overline{\underline 0}m\sqsubset \beta\rightarrow \exists\alpha[\alpha\in \mathcal{F}_\beta]]$ \textit{or}\\ $\forall \beta[\overline{\underline 0}m\sqsubset \beta\rightarrow \neg\exists\alpha[\alpha\in \mathcal{F}_\beta]]$. 
Both alternatives are false, as we see by considering $\beta = \underline{\overline0}m\ast\underline 1$ (for the first alternative), and $\beta=\underline 0$ (for the second one).\end{proof}

\begin{lemma}[Brouwer's Continuity Principle extends to spreads]\label{L:bcpspreads} $\;\;$\\Let $\beta$ be given such that $Spr(\beta)$. Then, for all $R\subseteq \mathcal{N}\times\omega$,\\ if $\forall \alpha\in\mathcal{F}_\beta\exists n[\alpha Rn]$, then $\forall \alpha\in\mathcal{F}_\beta\exists m \exists n\forall \gamma\in\mathcal{F}_\beta[\overline\alpha m \sqsubset\gamma\rightarrow \gamma Rn]$.\end{lemma}

\begin{proof} Assume: $Spr(\beta)$. If $\beta(\langle\;\rangle)\neq 0$, then $\mathcal{F}_\beta=\emptyset$ and there is nothing to prove.

 Assume   $\beta(\langle\;\rangle)=0$.
Define $\sigma$ such that $\sigma(\langle\;\rangle)=\langle\;\rangle$ and, for all $s$, for all $n$, \begin{enumerate}\item if $\beta(s\ast\langle n\rangle)=0$, then $\sigma(s\ast\langle n\rangle)=s\ast\langle n\rangle$, and, \item if $\beta(s\ast\langle n \rangle)\neq 0$, then $\sigma(s\ast\langle n \rangle)=\sigma(s)\ast\langle \mu p[\beta\bigl(\sigma(s)\ast\langle p\rangle\bigr)=0]\rangle$.\end{enumerate} Note: $\forall s[\beta\bigl(\sigma(s)\bigr)=0]$ and $\forall s\forall t[s\sqsubset t\rightarrow \sigma(s)\sqsubset\sigma(t)]$.
 
 Define $\rho:\mathcal{N}\rightarrow\mathcal{N}$ such that $\forall \alpha\forall n[ \sigma(\overline \alpha n)\sqsubset \rho|\alpha]$. 
 
 Note: $\forall \alpha[\rho|\alpha\in\mathcal{F}_\beta]\;\wedge\; \forall \alpha \in \mathcal{F}_\beta[\rho|\alpha=\alpha]$.
 
 The function $\rho$ is called a \textit{retraction} of $\mathcal{N}$ onto $\mathcal{F}_\beta$. 
 
 \medskip
 Now assume: $\forall \alpha\in\mathcal{F}_\beta\exists n[\alpha Rn]$. Conclude: $\forall \alpha\exists n[(\rho|\alpha)R n]$. 
 
 Let $\alpha$ in $\mathcal{F}_\beta$ be given.  
  Using Axiom \ref{ax:bcp},  find $m,n$ such that \\$\forall \gamma[\overline \alpha m\sqsubset \gamma\rightarrow (\rho|\gamma)Rn]$.
 Conclude: $\forall \gamma\in\mathcal{F}_\beta[\overline \alpha m\sqsubset \gamma\rightarrow \gamma R n]$.
 
 We thus see: $\forall \alpha \in \mathcal{F}_\beta\exists m\exists n\forall \gamma \in \mathcal{F}_\beta[\overline \gamma m\sqsubset \alpha\rightarrow \gamma Rn]$.
 \end{proof}
 
 Recall that, for all $\alpha, \beta$, $\alpha\;\#\;\beta\leftrightarrow \alpha\perp\beta\leftrightarrow \exists n[\alpha(n)\neq \beta(n)],$
 and  \\$\alpha=\beta\leftrightarrow \forall n[\alpha(n)=\beta(n)]\leftrightarrow \neg(\alpha\;\#\;\beta)$, and $\alpha\neq \beta \leftrightarrow \neg \forall n[\alpha(n)=\beta(n)].$

 The constructive \textit{apartness relation} $\#$ is more useful than the negative \textit{inequality relation} $\neq$.
 
  \textit{Markov's Principle}, in the form: $\forall \alpha[\neg\neg\exists n[\alpha(n)=0]\rightarrow\exists n[\alpha(n)=0]]$\footnote{A.A.~Markov enuntiated this principle for \textit{primitive recursive} $\alpha$ only.},  is equivalent to the statement  that the two relations coincide: $\forall \alpha\forall \beta[\alpha\neq \beta\rightarrow \alpha\;\#\;\beta]$.
   
    The intuitionistic mathematician does not accept Markov's Principle. 
 \begin{definition} We let 
$AP=AP(\mathsf{x,y})$ be the formula $\mathsf{\forall z[ \neg(z=x)\;\vee\;\neg(z=y)]}$.\end{definition}
 The following theorem reformulates a well-known fact.
\begin{theorem}[Apartness is definable]\label{T:apartdef} For all $ \beta$ such that $Spr(\beta)$, \\for all $\alpha, \delta$ in $\mathcal{F}_\beta$, 
$\alpha\;\#\;\delta$ if and only if $(\mathcal{F}_\beta,=)\models AP[\alpha, \delta]$. \end{theorem}

\begin{proof}First, assume $\alpha\;\#\;\delta$. Find $n$ such that $\overline \alpha n \neq\overline \delta n$. Note: for every $\gamma$ in $\mathcal{F}_\beta$, \textit{either}: $\overline \gamma n \neq \overline \alpha n$ and $\gamma\;\#\;\alpha$, \textit{or}: $\overline \gamma n\neq\delta n$ and $\gamma\;\#\;\delta$.  Conclude: $(\mathcal{F}_\beta,=)\models AP[\alpha, \delta]$. 

\smallskip

Next, assume $(\mathcal{F}_\beta,=)\models AP[\alpha, \delta]$, that is $\forall \gamma\in \mathcal{F}_\beta[\gamma \neq \alpha\;\vee\;\gamma\neq \delta]$. 

Applying Lemma \ref{L:bcpspreads}, find $m$ such that \textit{either} $\forall \gamma \in \mathcal{F}_\beta[\overline \alpha m\sqsubset \gamma \rightarrow \gamma \neq \alpha]$ \textit{or} $\forall \gamma \in \mathcal{F}_\beta[\overline \alpha m\sqsubset \gamma \rightarrow \gamma \neq \delta]$. The first alternative is clearly wrong (take $\gamma:=\alpha$). The second alternative implies:
$\overline\alpha m \perp \delta$ (if  $\overline \alpha m \sqsubset \delta$, one could take $\gamma:=\delta$), and thus: $\alpha\;\#\;\delta$.    \end{proof}

\begin{definition}\label{D:Tinfplus} $\;$

For each $n$, we let $\psi_n^+$ be the sentence $\mathsf{\exists x_0\exists x_1\ldots \exists x_n}[\bigwedge_{i<j<n}AP(\mathsf{x}_i,\mathsf{x}_j)]$.

 $T_{inf}^+:=EQ\cup\{\psi_n^+\mid n\in \omega\}$.\end{definition}

$\psi_n^+$ expresses that a set has at least $n+1$ elements that are mutually apart.

Every model of $T^+_{inf}$ realizes $T_{inf}$. In the second part of the paper we will meet a structure that realizes $T_{inf}$ but not $T^+_{inf}$, see Theorem \ref{T:vitalinoapartness} in Section
\ref{S:vitali}. 

The theory  $T^+_{inf}\cup \{\forall\mathsf{x\forall y[x=y\;\vee\;\neg(x=y)]}\}$ has only one complete extension, the  same as the one and only complete extension of $T_{inf}\cup \{\forall\mathsf{x\forall y[x=y\;\vee\;\neg(x=y)]}\}$, see Theorem \ref{T:classicalcompleteness}. 

\section{Spreads with a decidable equality}

 \begin{definition} We let $D=D(\mathsf{x})$ be the formula: $\mathsf{\forall y[x=y\;\vee\;\neg(x=y)]}$. 
\end{definition}

\begin{definition} Assume $Spr(\beta)$ and  
$\alpha\in \mathcal{F}_\beta$. 

$\alpha$ is an \emph{isolated} point of $\mathcal{F}_\beta$ if and only if $\exists n\forall \gamma \in \mathcal{F}_\beta[\overline \alpha n \sqsubset \gamma\rightarrow \alpha = \gamma]$, or, equivalently, $\exists n\forall s[\bigl(\overline \alpha n\sqsubset s\;\wedge\;\beta(s)=0\bigr)\rightarrow s\sqsubset \alpha]$.   

$\alpha$ is a \emph{decidable} point of $\mathcal{F}_\beta$ if and only if $\forall \gamma\in\mathcal{F}_\beta[\alpha =\gamma\;\vee\;\neg(\alpha =\gamma)]$, or, equivalently, $(\mathcal{F}_\beta,=)\models D[\alpha]$.

 $\mathcal{I}(\mathcal{F}_\beta)$ is the set of
  the isolated points of $\mathcal{F}_\beta$. \end{definition}
  
  Cantor called $\mathcal{I}(\mathcal{F}_\beta)$ the \textit{adherence} of $\mathcal{F}_\beta$.
 
\begin{lemma}\label{L:defisol} Assume $Spr(\beta)$. \begin{enumerate}[\upshape (i)] \item For each $\alpha$ in $\mathcal{F}_\beta$, $\alpha$ is an isolated point of $\mathcal{F}_\beta$ if and only if $\alpha$ is a decidable point of $\mathcal{F}_\beta$. 

\item $\mathcal{I}(\mathcal{F}_\beta)$ is a definable subset of $\mathcal{F}_\beta$. \end{enumerate} \end{lemma}

\begin{proof} (i)  
Let $\alpha$ be an isolated point of $\mathcal{F}_\beta$. 

Find $n$ such that $\forall \gamma \in \mathcal{F}_\beta[\overline \alpha n \sqsubset \gamma\rightarrow \alpha = \gamma]$.

Note: for each $\gamma$ in $\mathcal{F}_\beta$, \textit{either} $\overline \alpha n\sqsubset \gamma$ and  $\alpha =\gamma$, \textit{or} $\overline \alpha n\perp \gamma$ and $\alpha\neq \gamma$.  
 
  Conclude:  $\forall \gamma \in \mathcal{F}_\beta[\alpha=\gamma\;\vee\;\neg(\alpha=\gamma)]$, that is: $\alpha$ is a decidable point of $\mathcal{F}_\beta$. 

\smallskip
Now assume: $\alpha$ is a decidable point of $\mathcal{F}_\beta$, that is: $\forall \gamma \in \mathcal{F}_\beta[\alpha=\gamma\;\vee\;\neg(\alpha=\gamma)]$. 

Apply Lemma \ref{L:bcpspreads} and find $m$ such that \textit{either} $\forall \gamma \in \mathcal{F}_\beta[\overline \alpha m\sqsubset \gamma\rightarrow \alpha =\gamma]$ \textit{or} $\forall \gamma \in \mathcal{F}_\beta[\overline \alpha m\sqsubset \gamma\rightarrow \neg( \alpha =\gamma)]$. As the second alternative does not hold (take $\gamma =\alpha$), conclude: $\forall \gamma \in \mathcal{F}_\beta[\overline \alpha m\sqsubset \gamma\rightarrow \alpha =\gamma]$, and: $\alpha$ is an isolated point of $ \mathcal{F}_\beta$. 

\smallskip (ii) Using (i), note: $\mathcal{I}(\mathcal{F}_\beta) =\{\alpha\in \mathcal{F}_\beta\mid (\mathcal{F}_\beta, =)\models D[\alpha]\}$.  \end{proof}

\begin{definition}\label{D:limitpoint} Assume $Spr(\beta)$ and 
$\alpha \in \mathcal{F}_\beta$.

$\alpha$ is a \emph{limit point} of $\mathcal{F}_\beta$ if and only if $\forall n\exists \delta \in \mathcal{F}_\beta[ \overline \alpha n\sqsubset \delta\;\wedge\;\alpha\perp\delta]$, or, equivalently, $\forall n\exists s[\overline \alpha n \sqsubset s\;\wedge\; \beta(s) =0\;\wedge\;\overline \alpha n\perp s]$. 

$\mathcal{L}(\mathcal{F}_\beta)$ is the set of the limit points of $\mathcal{F}_\beta$. \end{definition}

Cantor called $\mathcal{L}(\mathcal{F}_\beta)$ the \textit{coherence} of $\mathcal{F}_\beta$.

 \begin{lemma}\label{L:limit} $\forall \beta[Spr(\beta)\rightarrow \mathcal{L}(\mathcal{F}_\beta)\subseteq\mathcal{F}_\beta\setminus \mathcal{I}(\mathcal{F}_\beta) ]$, that is:\\
   in all spreads,  every limit point is a non-isolated point.  
  \end{lemma} 
 
 \begin{proof} Obvious. \end{proof}

 \begin{theorem}\label{T:equivmarkov} The following are equivalent: \begin{enumerate}[\upshape (i)] \item Markov's Principle: $\forall\alpha[\neg\neg\exists n[\alpha(n)=0]\rightarrow \exists n[\alpha(n)=0]]$. \item $\forall \beta[Spr(\beta)\rightarrow \mathcal{F}_\beta\setminus \mathcal{I}(\mathcal{F}_\beta)\subseteq \mathcal{L}(\mathcal{F}_\beta)]$, that is:\\
   in all spreads,  every non-isolated point is a limit point.  \end{enumerate} \end{theorem}
 
 \begin{proof} (i) $\Rightarrow$ (ii). Let $\beta$ be given such that $Spr(\beta)$. Assume $\alpha$ is not an isolated point of $\mathcal{F}_\beta$, that is:  $\neg \exists n\forall s[\bigl(\overline\alpha n\sqsubset s\;\wedge\;\beta(s)=0\bigr)\rightarrow s\sqsubset\alpha]$.  
 
 Let $n$ be given. Define $\delta$ such that $\forall s[\delta(s)=0 \leftrightarrow (\overline \alpha n\sqsubset s\;\wedge\;\beta(s) =0 \;\wedge\; s\perp \alpha)]$. Then $\neg\forall s[\delta(s)\neq 0]$ and: $\neg\neg\exists s[\delta(s)=0]$. 
 
 Using \textit{Markov's Principle}, we  conclude: $\exists s[\delta(s)=0]$. 
 
 We thus see: 
  $\forall n \exists s[\overline \alpha s\sqsubset s \;\wedge\;\beta(s)=0 \;\wedge\; s\perp\alpha]$,  and: $\alpha$ is a limit point of $\mathcal{F}_\beta$.

  \smallskip 
  
 (ii) $\Rightarrow$ (i).   Let us assume: $\forall \beta[Spr(\beta)\rightarrow \mathcal{F}_\beta\setminus \mathcal{I}(\mathcal{F}_\beta)\subseteq \mathcal{L}(\mathcal{F}_\beta)]$, 
  
  Let $\alpha$ be given such that $\neg\neg\exists n[\alpha(n)=0]$.
  \\Define $\beta$ such that 
  $\forall s[\beta(s)=0\leftrightarrow \forall m<length(s)[s(m)\neq 0\rightarrow \exists n\le m[\alpha(n)=0]]]$.  
  
  Note: $Spr(\beta)$ and $\underline  0 \in \mathcal{F}_\beta$, and:  if $\exists n[\alpha(n)=0]$, then $\underline 0$ is a limit point of $\mathcal{F}_\beta$.
  
  Conclude: 
  if $\underline 0$ is an isolated point of $\mathcal{F}_\beta$, then $\neg\exists n[\alpha(n)=0]$.
  
 As $\neg\neg\exists n[\alpha(n)=0]$, 
 conclude: $\underline 0$ is not an isolated point of $\mathcal{F}_\beta$. 
 
 By our assumption,  $\underline 0$ thus is a limit point of $\mathcal{F}_\beta$.
 
  Find $s$ such that $\beta(s)=0$ and $s \perp \underline 0$. Conclude: $\exists n\le length(s)[\alpha(n)=0]$.
  
  Conclude: $\forall\alpha[\neg\neg\exists n[\alpha(n)=0]\rightarrow\exists n[\alpha(n)=0]]$, that is: Markov's Principle. 
  \end{proof}
  \smallskip
  We thus see that the converse of Lemma \ref{L:limit}, being equivalent to Markov's Principle, is not an intuitionistic theorem. 
  
  \smallskip We could not answer the question if, in general, $\mathcal{L}(\mathcal{F}_\beta)$ is a definable subset of $(\mathcal{F}_\beta, =)$. In some special cases, however, it is, and the following definition is useful.
  
  \begin{definition}\label{D:transparent} Assume $Spr(\beta)$. $\mathcal{F}_\beta$ is called \emph{transparent} if and only if there exists $\gamma$ such that $Spr(\gamma)$ and $\mathcal{F}_\gamma = \mathcal{L}(\mathcal{F}_\beta)$ and $\forall \alpha \in \mathcal{F}_\beta[\exists n[\gamma(\overline \alpha n)\neq 0]\rightarrow \alpha \in \mathcal{I}(\mathcal{F}_\beta)]$.
   \end{definition}
   
   Note that, for each $\beta$ such that $Spr(\beta)$, if $\mathcal{F}_\beta$ is transparent, then \\$\mathcal{F}_\beta\setminus \mathcal{I}(\mathcal{F}_\beta)\subseteq \mathcal{L}(\mathcal{F}_\beta)$. 
    The statement that every spread $\mathcal{F}_\beta$ is transparent thus is seen to imply Markov's Principle. 
    
    In Section \ref{S:moreundec} we will see many examples of transparent spreads.
    
    The fact that not every spread is a transparent spread is one of the reasons that Brouwer did not succeed in finding a nice intuitionistic version of Cantor's Main Theorem\footnote{Cantor's Main Theorem nowadays is called the Perfect Set Theorem: \textit{every closed subset of $\mathcal{N}$ is the union of a perfect set and an at most countable set.}}, see \cite{brouwer19}. 
  
\begin{definition}\label{D:functionspreads} Let $\beta$ satisfy $Spr(\beta)$ and let $\varphi$ be given. 

\smallskip
We define: $\varphi:\mathcal{F}_\beta\rightarrow \omega$ if and only if $\forall \alpha \in \mathcal{F}_\beta \exists p[\varphi(\overline\alpha p)\neq 0]$.

If $\varphi:\mathcal{F}_\beta \rightarrow \omega$, then we define, for each $\alpha$ in $\mathcal{F}_\beta$, $\varphi(\alpha)$ as the number $z$ such that $\varphi(\overline \alpha q)=z+1$, where $q=\mu p[\varphi(\overline \alpha p)\neq 0]$. 

\smallskip
We define: $\varphi$ is an injective map from $\mathcal{F}_\beta$ into $\omega$, notation: $\varphi:\mathcal{F}_\beta \hookrightarrow\omega$,\\ if and only if $\varphi:\mathcal{F}_\alpha\rightarrow \omega$ and $\forall \alpha\in\mathcal{F}_\beta\forall \delta\in\mathcal{F}_\beta[\alpha\;\#\;\delta\rightarrow \varphi(\alpha)\neq\varphi(\delta)]$.

\smallskip
We define: $\varphi: \mathcal{F}_\beta\rightarrow \mathcal{N}$ if and only if $\forall n[\varphi^n:\mathcal{F}_\beta\rightarrow \omega]$.

If $\varphi:\mathcal{F}_\beta\rightarrow \mathcal{N}$, then we define, for each $\alpha$ in $\mathcal{F}_\beta$, $\varphi|\alpha$ as the element $\delta$ of $\mathcal{N}$ such that $\forall n[\delta(n)=\varphi^n(\alpha)]$.

\smallskip
We define: $\varphi$ is an injective map from $\mathcal{F}_\beta$ into $\mathcal{N}$, notation: $\varphi:\mathcal{F}_\beta \hookrightarrow\mathcal{N}$,\\ if and only if $\varphi:\mathcal{F}_\alpha\rightarrow \mathcal{N}$ and $\forall \alpha\in\mathcal{F}_\beta\forall \delta\in\mathcal{F}_\beta[\alpha\;\#\;\delta\rightarrow \varphi|\alpha\;\#\;\varphi|\delta]$.

\smallskip For every $\mathcal{X}\subseteq \mathcal{N}$, $\mathcal{F}_\beta$ \emph{embeds into} $\mathcal{X}$ if and only if there exists an injective map from $\mathcal{F}_\beta$ into $\mathcal{X}$. 
\end{definition}

The following axiom is, at least at first sight,  a little bit stronger than Brouwer's Continuity Principle.

\begin{axiom}[First Axiom of Continuous Choice]\label{ax:fcc} For all $R\subseteq \mathcal{N}\times\omega$,\\ if $\forall \alpha \exists n[\alpha Rn]$, then $\exists \varphi:\mathcal{N}\rightarrow\omega\forall \alpha[\alpha R\varphi(\alpha)]$. \end{axiom}

\begin{lemma}[The First Axiom of Continuous Choice extends to spreads]\label{L:fccspreads}$\;\;$\\ Let $\beta$ be given such that $Spr(\beta)$. Then, for all $R\subseteq \mathcal{F}_\beta\times \omega$, \\if $\forall \alpha \in \mathcal{F}_\beta\exists n[\alpha R n]$, then $\exists\varphi: \mathcal{F}_\beta\rightarrow \omega\forall\alpha\in\mathcal{F}_\beta[\alpha R\varphi(\alpha)]$. \end{lemma}

\begin{proof}Assume: $Spr(\beta)$ and   $\beta(\langle\;\rangle)=0$.
As in the proof of Lemma \ref{L:bcpspreads}, define $\rho:\mathcal{N}\rightarrow\mathcal{F}_\beta$ such that  $\forall \alpha[\rho|\alpha\in\mathcal{F}_\beta]\;\wedge\; \forall \alpha \in \mathcal{F}_\beta[\rho|\alpha=\alpha]$.

 \medskip
 Now assume $\forall \alpha\in\mathcal{F}_\beta\exists n[\alpha Rn]$. Conclude: $\forall \alpha\exists n[(\rho|\alpha)R n]$.

  Applying Axiom \ref{ax:fcc}, find $\varphi:\mathcal{N}\rightarrow\omega$ such that $\forall \gamma[ (\rho|\gamma)R\varphi(\gamma)]$.
 
 Conclude: $\varphi:\mathcal{F}_\beta\rightarrow \omega$ and $\forall \gamma\in\mathcal{F}_\beta[\gamma R\varphi(\gamma)]$.
 \end{proof}
\begin{theorem}\label{T:decsprintoomega} Assume $Spr(\beta)$. $(\mathcal{F}_\beta,=)\models \mathsf{\forall x}[D(\mathsf{x})]$ if and only if $\exists \varphi[\varphi:\mathcal{F}_\beta\hookrightarrow \omega]$. \end{theorem}

\begin{proof} First assume: $(\mathcal{F}_\beta,=)\models\mathsf{\forall x}[D(\mathsf{x})]$. Then, by Lemma \ref{L:defisol},\\ $\forall \alpha\in\mathcal{F}_\beta\exists n\forall \gamma\in\mathcal{F}_\beta[\overline\alpha n\sqsubset \gamma\rightarrow \alpha =\gamma]$. Using Lemma \ref{L:fccspreads}, find $\varphi:\mathcal{F}_\beta\rightarrow \omega$ such that $\forall \alpha\in\mathcal{F}_\beta\forall \gamma\in\mathcal{F}_\beta[\overline\alpha \varphi(\alpha)\sqsubset \gamma\rightarrow \alpha =\gamma]$. Define $\psi:\mathcal{F}_\beta\rightarrow \omega$ such that $\forall \alpha \in \mathcal{F}_\beta[\psi(\alpha)=\overline \alpha \varphi(\alpha)]$. Clearly, $\psi:\mathcal{F_\beta}\hookrightarrow \omega$.

\medskip Now assume: $\varphi:\mathcal{F}_\beta\hookrightarrow \omega$. Note: $\forall \alpha\in\mathcal{F}_\beta\forall \delta\in \mathcal{F}_\beta[\alpha = \delta\leftrightarrow \varphi(\alpha)=\varphi(\delta)]$. \\Also: $\forall \alpha\in\mathcal{F}_\beta\forall \delta\in \mathcal{F}_\beta[\varphi(\alpha)=\varphi(\delta)\;\vee\;\neg\bigl(\varphi(\alpha)=\varphi(\delta)\bigr)]$. \\Therefore: $\forall \alpha\in\mathcal{F}_\beta\forall \delta\in \mathcal{F}_\beta[\alpha = \delta\;\vee\; \neg(\alpha =\delta)]$.
Conclude:   $(\mathcal{F}_\beta,=)\models \mathsf{\forall x}[D(\mathsf{x})]$.
\end{proof} 

\begin{definition} Assume $Spr(\beta)$. $\mathcal{F}_\beta$ is \emph{enumerable} if and only if either $\mathcal{F}_\beta=\emptyset$ or $\exists \delta[\forall n[\delta^n \in \mathcal{F}_\beta]\;\wedge\;\forall \alpha \in \mathcal{F}_\beta\exists n[\alpha =\delta^n]]$. \end{definition}

\begin{lemma}\label{L:enumer} Assume $Spr(\beta)$. $\mathcal{F}_\beta$ is enumerable if and only if $\exists \varphi[\varphi:\mathcal{F}_\beta\hookrightarrow \omega]$. \end{lemma}

\begin{proof} Assume $\mathcal{F}_\beta$ is enumerable and $\beta(\langle\;\rangle)=0$.\\ 
Find $\delta$ such that $\forall n[\delta^n\in \mathcal{F}_\beta]$ and $\forall \alpha \in \mathcal{F}_\beta\exists n[\alpha =\delta^n]$. \\Using Lemma \ref{L:fccspreads}, find $\varphi:\mathcal{F}_\beta\rightarrow \omega$ such that $\forall \alpha \in \mathcal{F}_\beta[\alpha=\delta^{\varphi(\alpha)}]$.\\ Note: $\varphi:\mathcal{F}_\beta \hookrightarrow \omega$. 

\medskip Now assume: $\varphi:\mathcal{F}_\beta \hookrightarrow \omega$. 

We make a preliminary observation.

Let $s, n$ be given such that  $\beta(s) =0$ and $\varphi(s)=n+1$ and $\forall t\sqsubset s[\varphi(t)=0]$. \\Note: $\forall \alpha \in \mathcal{F}_\beta[s\sqsubset \alpha\rightarrow \varphi(\alpha)=n]$ and, therefore: \\$\forall \alpha \in \mathcal{F}_\beta\forall \delta \in \mathcal{F}_\beta[(s\sqsubset \alpha\;\wedge\; s\sqsubset \delta)\rightarrow \alpha =\delta]$. 

\smallskip Now let $\gamma$ be the element of $\mathcal{F}_\beta$ satisfying $\forall n[\gamma(n):=\mu p[\beta(\overline \gamma n\ast \langle p\rangle)=0]]$.
\\ Define $\delta$ such that, for all $s$, \textit{if} $\beta(s) =0$ and $\varphi(s)\neq 0$ and $\forall t\sqsubset s[\varphi(t)=0]$, then $s\sqsubset \delta^s$ and $\delta^s \in \mathcal{F}_\beta$, and \textit{if not}, then $\delta^s=\gamma$. 

Note: $\forall s[\delta^s \in \mathcal{F}_\beta]$ and $\forall \alpha \in \mathcal{F}_\beta\exists s[\alpha =\delta^s]$. 
\end{proof}

\begin{corollary}\label{C:decenum} Assume $Spr(\beta)$. \\$(\mathcal{F}_\beta, =)\models \forall\mathsf{x}[D(\mathsf{x})]$ if and only if $\mathcal{F}_\beta$ is enumerable. \end{corollary}

\begin{proof} Use Theorem \ref{T:decsprintoomega} and Lemma \ref{L:enumer}. \end{proof}
\section{Spreads with exactly one undecidable point}
\begin{definition} We let $\tau_2$ be the element of $\mathcal{C}$ satisfying: \\$\forall s[\tau_2(s)=0\leftrightarrow \forall i<length(s)[s(i)<2\;\wedge\; \bigl(i+1<length(s)\rightarrow s(i)\le s(i+1)\bigr)]]$. \\We  define: $\mathcal{T}_2:=\mathcal{F}_{\tau_2}$.

 \end{definition}
 
 Note: $\tau_2$ is a spread-law and $\mathcal{T}_2$ is a spread.

 \smallskip
 Let us take a closer look at $\mathcal{T}_2$.

 Observe: $\forall \alpha[\alpha\in\mathcal{T}_2\leftrightarrow \forall i[\alpha(i)\le\alpha(i+1)<2]]$.
 
 For each $n$, we define $n^\ast:= \underline{\overline 0}n\ast\underline 1$. 
 
   \smallskip The infinite sequence $\underline 0, 0^\ast, 1^\ast, 2^\ast, \ldots$ is a  list of elements of $\mathcal{T}_2$ and a classical mathematician might think it is the list of all elements of $\mathcal{T}_2$. The intuitionistic mathematician knows better. He  defines $\alpha$ in $\mathcal{T}_2$ such that 
   $$\forall n[\alpha(n) =1\leftrightarrow \exists k\le n\forall i<99[d(k+i)=9]],$$ where $d:\mathbb{N}\rightarrow \{0,1,\ldots,9\}$ is the decimal expansion of $\pi$. As yet, one has  no proof of the statement `$\alpha =\underline 0$', as this statement implies: $\forall k\exists i<99]d(k+i)=9]$. As yet,  one also has  no proof of the statement: `$\exists n[\alpha= n^\ast]$' as this statement implies: $\exists n\forall i<99[d(n+i)=9]$. The statement that $\alpha$ occurs in the above list is a reckless one.

  For each $n$, $n^\ast$ is an isolated and a decidable point of $\mathcal{T}_2$, and   $\underline 0$ is a non-isolated and an undecidable point of $\mathcal{T}_2$.
   It follows, by Lemma \ref{L:defisol} and Corollary \ref{C:decenum}, that $\mathcal{T}_2$ is not an enumerable spread.
  In particular, the statement that the list $\underline 0, 0^\ast, 1^\ast, 2^\ast, \ldots$ is a complete list of the elements of $\mathcal{T}_2$, leads to a contradiction, as appears again from the following Theorem.

  \begin{theorem}\label{T:tau2} \begin{enumerate}[\upshape (i)] \item $\neg\forall\alpha\in\mathcal{T}_2[\alpha =\underline 0\;\vee\;\exists n[\alpha = n^\ast]]$. \item $\forall\alpha\in\mathcal{T}_2[\alpha \;\#\;\underline 0\rightarrow \exists n[\alpha = n^\ast]]$. \end{enumerate} \end{theorem}
  
  \begin{proof} (i) Assume $\forall\alpha\in\mathcal{T}_2[\alpha =\underline 0\;\vee\;\exists n[\alpha = n^\ast]]$. Using Lemma \ref{L:bcpspreads}, find $m,n$ such that \textit{either} $\forall\alpha\in \mathcal{T}_2[\underline{\overline 0}m\sqsubset\alpha\rightarrow \alpha=\underline 0]$ \textit{or} $\forall\alpha\in \mathcal{T}_2[\underline{\overline 0}m\sqsubset\alpha\rightarrow \alpha=n^\ast]$. Note that both alternatives are false.
  
  Conclude: $\neg\forall\alpha\in\mathcal{T}_2[\alpha =\underline 0\;\vee\;\exists n[\alpha = n^\ast]]$.
  
  \smallskip 
  
  (ii) Let $\alpha$ in $\mathcal{T}_2$ be given such that $\alpha\;\#\;\underline 0$. Define $n:=\mu m[\overline \alpha(m+1)\perp \underline 0]$. Note: $\overline \alpha(n+1)=\overline{\underline 0}n\ast\langle 1 \rangle$ and $\alpha =n^\ast$. \end{proof}
  
\begin{definition} Assume $Spr(\beta)$. $\mathcal{F}_\beta$ is \emph{almost-enumerable} if and only if either $\mathcal{F}_\beta=\emptyset$ or $\exists \delta[\forall n[\delta^n\in\mathcal{F}_\beta]\;\wedge\;\forall \alpha\in \mathcal{F}_\beta\forall \varepsilon \exists n[\overline \alpha \varepsilon(n) = \overline{\delta^n}\varepsilon(n)]]$. \end{definition}

This definition deserves some explanation. If $\mathcal{F}_\beta$ is almost-enumerable and inhabited, we are able to come forward with an infinite sequence $\delta^0, \delta^1, \ldots$ of elements of $\mathcal{F}_\beta$ such that, for every $\alpha$ in $\mathcal{F}_\beta$, every attempt $\varepsilon$ to prove that $\alpha$ is apart from all elements of the infinite sequence $\delta^0, \delta^1, \ldots$, ($\varepsilon$  expresses the guess: $\forall n[\overline \alpha\varepsilon(n)\perp \overline{\delta^n}\varepsilon(n)]$), will positively fail.

Almost-enumerable spreads are studied in \cite[Section 9]{veldman2019}, where they are called \textit{almost-countable located and closed subsets of $\mathcal{N}$.}

\begin{theorem} $\mathcal{T}_2$ is almost-enumerable. \end{theorem} \begin{proof} Define $\delta$ such that $\delta^0=\underline 0$ and, for each $n$. $\delta^{n+1}=n^\ast =\underline{\overline 0}n\ast\underline 1$. Note: $\forall n[\delta^n \in \mathcal{T}_2]$. Let $\varepsilon$ be given. If $\overline\alpha\varepsilon(0)=\overline{\delta^0}\varepsilon(0)$, we are done. If not, then $\alpha\perp\underline 0$ and we may determine $n$ such that $\alpha=\delta^{n+1}$ and $\overline \alpha \varepsilon(n+1)=\overline{\delta^{n+1}}\varepsilon(n+1)$. \end{proof}

\begin{axiom}[Second Axiom of Countable Choice]\label{ax:countable choice}$\;\;$ \\ For every $R\subseteq \mathbb{N}\times \mathcal{N}$, if $\forall n\exists \alpha[nR\alpha]$, then $\exists \alpha\forall n[n R \alpha^n]$. \end{axiom}
\begin{theorem} \begin{enumerate}[\upshape (i)] \item $(\mathcal{T}_2,=)\models\exists\mathsf{x}[\neg D(\mathsf{x})\;\wedge\;\forall \mathsf{y}[AP(\mathsf{x, y})\rightarrow D(\mathsf{y})]]$. 
\item For all $\beta$ such that  $Spr(\beta)$,\\ if $(\mathcal{F}_\beta,=)\models \exists\mathsf{x}[\neg D(\mathsf{x})\;\wedge\;\forall \mathsf{y}[AP(\mathsf{x, y})\rightarrow D(\mathsf{y})]]$, then $\mathcal{F}_\beta$ embeds into $\mathcal{T}_2$. \end{enumerate} \end{theorem}

\begin{proof} (i)  $\underline 0$ is not an isolated point of $\mathcal{T}_2$, and, therefore, not a decidable point of $\mathcal{T}_2$. 
Also, by Theorem \ref{T:tau2}(ii), $\forall \alpha \in \mathcal{T}_2[\alpha\;\#\;\underline 0\rightarrow \exists n[\alpha= n^\ast]]$, and, 
 for each $n$, for each $\alpha$ in $\mathcal{T}_2$, $\alpha = n^\ast\leftrightarrow \overline{\underline 0}n\ast\langle 1\rangle\sqsubset \alpha$, so one may decide: $\alpha = n^\ast$ or $\neg(\alpha = n^\ast)$, and: $n^\ast$ is a decidable point of $\mathcal{T}_2$.  We thus see:
 $(\mathcal{T}_2,=)\models\neg D(\mathsf{x})\;\wedge\;\forall \mathsf{y}[AP(\mathsf{x, y})\rightarrow D(\mathsf{y})][\underline 0]$, and are done. 
 
 \smallskip (ii)  Assume: $Spr(\beta)$ and $(\mathcal{F}_\beta,=)\models \exists\mathsf{x}[\neg D(\mathsf{x})\;\wedge\;\forall \mathsf{y}[AP(\mathsf{x, y})\rightarrow D(\mathsf{y})]]$.
 
 Find $\alpha$ in $\mathcal{F}_\beta$ such that $\alpha$ is not an isolated point of $\mathcal{F}_\beta$. 
 
 Note: for each $s$ such that $\beta(s)=0$, the set $\mathcal{F}_\beta\cap s:=\{\delta \in \mathcal{F}_\beta\mid s\sqsubset \delta\}$ is a spread, and, if $s\perp\alpha$, then $\mathcal{F}_\beta\cap s$ consists of isolated points of $\mathcal{F}_\beta\cap s$ only,  and thus, by Theorem \ref{T:decsprintoomega}, embeds into $\omega$.  
 
 Using Axiom \ref{ax:countable choice}, we find $\varphi$ such that, for each $s$, if $\beta(s)=0$ and there exist $n, i$ such that $s=\overline \alpha n\ast\langle i\rangle$ and $i\neq \alpha(n)$, then $\varphi^s:\mathcal{F}_\beta\cap s\hookrightarrow \omega$.  
 
 We now  define $\psi:\mathcal{F}_\beta\rightarrow \mathcal{T}_2$ such that $\psi|\alpha =\underline 0$ and, for each $\delta$ in $\mathcal{F}_\beta$, if $\delta\;\#\;\alpha$, then $\psi|\delta=\overline{\underline 0}\bigl(\overline \delta n, \varphi^{\overline \delta n}(\delta)\bigr)\ast \underline 1$ where $n:=\mu i[\overline \delta i \perp \alpha]$. \end{proof}
 
 \section{More and more undecidable points: the toy spreads}\label{S:moreundec}\begin{definition} For each $n$, we let $\tau_n$ be the element of $\mathcal{C}$ satisfying: \\$\forall s[\tau_n(s)=0\leftrightarrow \forall i<length(s)[s(i)<n\;\wedge\; \bigl(i+1<length(s)\rightarrow s(i)\le s(i+1)\bigr)]]$. \\We also define: $\mathcal{T}_n:=\mathcal{F}_{\tau_n}$.

 \end{definition}
 
 For each $n$, $\tau_n$ is a spread-law and $\mathcal{T}_n$  and 
 $\mathcal{T}_n=\{\alpha\mid \forall i[\alpha(i)\le\alpha(i+1)<n]\}$ is a spread. 
 
 \smallskip In this paper, the spreads $\mathcal{T}_0, \mathcal{T}_1, \ldots$ will be called the \textit{toy spreads}.  
 
 Note: $\mathcal{T}_0=\emptyset$ and $\mathcal{T}_1=\{\underline 0\}$.

 \begin{definition} For each $s\neq \langle\;\rangle$, we let $s^\dag$ be the element of $\mathcal{N}$ satisfying $s\sqsubset s^\dag$ and $\forall i\ge 
 length(s)[s^\dag(i)=s^\dag(i-1)]$.
  \end{definition}
  
  Note that, for each $n$, for each $s$, if $s\neq\langle\;\rangle$ and $\tau_n(s)=0$, then $s^\dag \in \mathcal{T}_n$.

 \begin{theorem} For each $n>0$, $\mathcal{T}_n$ is almost-enumerable. \end{theorem}

\begin{proof} Let $n>0$ be given. Define $\delta$ such that, for each $s$, \textit{if} $s\neq\langle \;\rangle$ and $\tau_n(s)=0$, then $\delta^s=s^\dag$, and \textit{if not}, then $\delta^s = \underline 0$.  

We claim: $\forall \alpha\in\mathcal{T}_n\forall \varepsilon \exists s[\overline \alpha\varepsilon(s)=\overline{\delta^s}\varepsilon(s)]$. 

We establish this claim by proving, for each $k<n$, \\$\forall \alpha\in\mathcal{T}_n[\exists i[\alpha(i)\ge k]\rightarrow\forall \varepsilon \exists s[\overline \alpha\varepsilon(s)=\overline{\delta^s}\varepsilon(s)]]$, and we do so by backwards induction, starting with the case $k=n-1$. 

\smallskip The case $k=n-1$ is treated  as follows.  If $\exists i[\alpha(i)=n-1]$, find\\ $i_0:=\mu i[\alpha(i)=n-1]$ and consider $s:=\overline \alpha(i_0+1)$. \\Note: $\alpha=s^\dag=\delta^s$ and, therefore, for every $\varepsilon$: $\overline \alpha\varepsilon(s)=\overline{\delta^s}\varepsilon(s)$. 

\smallskip Now assume $k<n-1$ is given such that   \\$\forall \alpha\in\mathcal{T}_n[\exists i[\alpha(i)\ge k+1]\rightarrow\forall \varepsilon \exists s[\overline \alpha\varepsilon(s)=\overline{\delta^s}\varepsilon(s)]]$.

We have to prove: $\forall \alpha\in\mathcal{T}_n[\exists i[\alpha(i)=k]\rightarrow\forall \varepsilon \exists s[\overline \alpha\varepsilon(s)=\overline{\delta^s}\varepsilon(s)]]$.  

Let $\alpha$ be given such that $\exists i[\alpha(i)=k]$. Let also $\varepsilon$ be given.

Define $i_0:=\mu i[\alpha(i)=k]$ and define $s:=\overline\alpha(i_0+1)$.

 There are two cases to consider.

\textit{Case (i)}: $\overline \alpha\varepsilon(s)=\overline{s^\dag}\varepsilon(s)=\overline{\delta^s}\varepsilon(s)$. We are done.

\textit{Case (ii)}: $\overline \alpha\varepsilon(s)\perp\overline{s^\dag}\varepsilon(s)$. Then $\exists i<\varepsilon(s) [\alpha(i)\ge k+1]$. 

Using the induction hypothesis, we conclude: $\exists s[\overline \alpha\varepsilon(s)=\overline{\delta^s}\varepsilon(s)]$. \end{proof}

\newpage
\begin{theorem}\label{T:taun} $\;$
\begin{enumerate}[\upshape (i)] \item For each $n$, for all $\alpha$ in $\mathcal{T}_n$, $\alpha \in \mathcal{I}(\mathcal{T}_n)$ if and only if $\exists m[\alpha(m)+1=n]$. \item For each $n$, $\mathcal{T}_{n+1}\setminus \mathcal{I}(\mathcal{T}_{n+1}) = \mathcal{T}_n=\mathcal{L}(\mathcal{T}_{n+1})$. \item For each $n$, $\mathcal{T}_n=\{\alpha \in \mathcal{T}_{n+1}\mid (\mathcal{T}_{n+1},=)\models \neg D[\alpha]\}$.  \end{enumerate}\end{theorem}

\begin{proof} The proof uses Lemma \ref{L:defisol} and is left to the reader.\end{proof}

\begin{definition} We define an infinite sequence $D_0, D_1, \ldots$ of formulas, as follows.

$D_0:= \mathsf{\forall y[x=y\;\vee\;\neg(x=y)}]$,

$D_1:=\neg D_0(\mathsf{x}) \;\wedge\; \forall \mathsf{y}[\neg D_0(\mathsf{y})\rightarrow \bigl(\mathsf{x=y\;\vee\;\neg(x=y)\bigr)}]$,

$D_2:=\neg D_0(\mathsf{x})\;\wedge\; \neg D_1(\mathsf{x})\;\wedge\; \forall \mathsf{y}[\bigl(\neg D_0(\mathsf{y})\;\wedge\;\neg D_1(\mathsf{y})\bigr)\rightarrow \bigl(\mathsf{x=y\;\vee\;\neg(x=y)\bigr)}]$,

and, more generally for each $m>0$,

$D_{m}:=\bigwedge_{i<m}\neg D_i(\mathsf{x}) \;\wedge\; \forall \mathsf{y}[\bigl(\bigwedge_{i<m}\neg D_i(\mathsf{y})\bigr)\rightarrow \bigl(\mathsf{x=y\;\vee\;\neg(x=y)\bigr)}]$.

\medskip We also define, 
for each $m>0$, sentences $\psi_m$ and $\rho_m$, as follows:

$\psi_{m}:=\exists\mathsf{x}[ D_{m}(\mathsf{x})]$ and $\rho_m:=\exists\mathsf{x}[D_m(\mathsf{x})\;\wedge\;\forall \mathsf{y}[D_m(\mathsf{y})\rightarrow \mathsf{y=x}]]$.
\end{definition}

\begin{definition} Assume $Spr(\beta)$. $\alpha$ in $\mathcal{F}_\beta$ is a \emph{limit point of order $0$} of $\mathcal{F}_\beta$ if and only if $\alpha$ is an isolated point of $\mathcal{F}_\beta$.

For each $m$, $\alpha$ is a \emph{limit point of order $m+1$} of $\mathcal{F}_\beta$ if and only if, for each $p$, there exists a limit point $\gamma$ of order $m$ such that $\overline \alpha p\sqsubset \gamma$ and $\alpha \perp \gamma$.  \end{definition}

Assume $n>0$ and $\alpha\in\mathcal{T}_n$. Note the following:

\begin{enumerate}\item  $(\mathcal{T}_n, =)\models D_0[\alpha]$ if and only if  $\alpha$ is an isolated point  of $\mathcal{T}_n$ if and only if either: $n=1$ or: $n>1$ and  $\exists p[\alpha(p) = n-1]$.   
\item  $(\mathcal{T}_n, =)\models \neg D_0[\alpha]$ if and only if $\alpha$ is a limit point (of order 1) of $\mathcal{T}_n$ if and only if $n>1$ and $\alpha \in \mathcal{T}_{n-1}$. \item  $(\mathcal{T}_n, =)\models  D_1[\alpha]$ if and only if $\alpha$ is an isolated point among the  limit points (of order 1) of $\mathcal{T}_n$ if and only if   $n>1$ and $\alpha\in \mathcal{T}_{n-1}$ and $\exists p[\alpha(p) =n-2]$. \item
 $(\mathcal{T}_n, =)\models \neg D_0\;\wedge\;\neg D_1[\alpha]$, if and only if $\alpha$ is a limit point of order $2$  of $\mathcal{T}_n$ if and only if $n>2$ and $\alpha \in \mathcal{T}_{n-2}$.  
 
 \item
 For  each $m>0$,   $(\mathcal{T}_n, =)\models  D_{2}[\alpha]$ if and only if $\alpha$ is an isolated point among the limit points of order $2$ if and only if  $n>2$ and $\alpha\in \mathcal{T}_{n-2}$ and $\exists p[\alpha(p)= n-3]$.
\item For each $m>0$,  $(\mathcal{T}_n, =)\models \bigwedge_{i<m} \neg D_i[\alpha]$ if and only if $\alpha$ is a limit point of order $m$  of $\mathcal{T}_n$  if and only if $n>m$ and  $\alpha \in \mathcal{T}_{n-m}$.
\item For  each $m>0$,   $(\mathcal{T}_n, =)\models  D_{m}[\alpha]$ if and only if $\alpha$ is an isolated point among the limit points of order $m$ if and only if  $n>m$ and $\alpha\in \mathcal{T}_{n-m}$ and $\exists p[\alpha(p)= n-m-1]$.  \item For each $m>0$, $\mathcal{T}_n\models \psi_m$ if and only if $\mathcal{T}_n$ contains an isolated point of $\mathcal{T}_{n-m}$ if and only if $n>m$. \item For each $m>0$, $\mathcal{T}_n\models \rho_m$ if and only if $\mathcal{T}_n$ contains exactly one isolated point of $\mathcal{T}_{n-m}$ if and only if $\mathcal{T}_{n-m}=\{\underline 0\}$ if and only if $n=m+1$.
\end{enumerate}

After these preliminary observations, the following Theorem is easy to understand:
\begin{theorem}\label{T:psirho}$\;$ \begin{enumerate}[\upshape(i)]\item For each $n$, $\mathcal{T}_n$ is a transparent\footnote{See Definition \ref{D:transparent}.} spread and,\\ if $n>0$, then $\mathcal{I}(\mathcal{T}_n)= \{\alpha \in \mathcal{T}_n\mid\exists p[\alpha(p)+1=n]\}$ and $\mathcal{L}(\mathcal{T}_n)=\mathcal{T}_{n-1}$. \item For all $n$, for all $m>0$, $\mathcal{T}_n=\{\alpha \in \mathcal{T}_{n+m}\mid(\mathcal{T}_{n+m},=)\models \bigwedge_{i<m} \neg D_i[\alpha]\}$.  \item  For all $m$, $\{\underline 0\}=\mathcal{T}_1=\{\alpha \in \mathcal{T}_{m+1}\mid(\mathcal{T}_{m+1},=)\models   \bigwedge_{i<m} \neg D_i[\alpha]\}$.\item For all $n>0$, for all $m>0$, $(\mathcal{T}_n, =)\models \psi_m$ if and only if 
$m+1\le n$. \item For all $n>0$, for all $m>0$, $(\mathcal{T}_n, =)\models \rho_m$ if and only if 
$m+1=n$.\end{enumerate}\end{theorem}

\begin{proof} Use the preliminary observations preceding this Theorem.
\end{proof}

\begin{corollary} For all $n, m$, if $n\neq m$, then there exists a sentence $\psi$ such that $(\mathcal{T}_m,=)\models \psi$ and $(\mathcal{T}_n, =)\models\neg \psi$. \end{corollary}

\section{Finite and infinite sums of toy spreads}
\subsection{A main result}\begin{definition} Assume $Spr(\beta), Spr(\gamma)$.
 
  We define: $\mathcal{F}_\beta\uplus\mathcal{F}_\gamma:=\{\langle 0\rangle\ast\delta\mid\delta\in\mathcal{F}_\beta\}\cup\{\langle 1\rangle\ast\delta\mid\delta\in\mathcal{F}_\gamma\}$. 
  
  For each $m$, we define: $m\otimes\mathcal{F}_\beta:=\{\langle i\rangle\ast\delta\mid i<m, \delta\in\mathcal{F}_\beta\}$.
  
  We also define: $\omega\otimes\mathcal{F}_\beta:=\{\langle i\rangle\ast\delta\mid i\in\omega, \delta\in\mathcal{F}_\beta\}$.

  \smallskip
 Note that $\mathcal{F}_\beta\uplus\mathcal{F}_\gamma$, $m\otimes\mathcal{F}_\beta$ and $\omega\otimes\mathcal{F}_\beta$ are spreads again.
 
 \smallskip We also define, for all $m,n>0$, sentences $\psi_m^n$ and $\rho_m^n$, as follows:
 
 $\psi^n_{m}:=\exists\mathsf{x}_0\exists\mathsf{x_1}\ldots \exists\mathsf{x}_{n-1}[\bigwedge_{i<j<n}[AP(\mathsf{x}_i, \mathsf{x}_j) \;\wedge\;\bigwedge_{i<n}\bigwedge_{j<m}\neg D(\mathsf{x}_j)]$.

 and $\rho^n_m:=\exists\mathsf{x}_0\exists\mathsf{x_1}\ldots \exists\mathsf{x}_{n-1}[\bigwedge_{i<j<n}[AP(\mathsf{x}_i, \mathsf{x}_j) \;\wedge\;\bigwedge_{i<n}\bigwedge_{j<m}\neg D_j(\mathsf{x}_i)\;\wedge\\\forall \mathsf{z}[\bigwedge_{j<m} \neg D_j(\mathsf{z})\rightarrow \bigvee_{i<n} \mathsf{z}=\mathsf{x}_i] ]$.
 
 \end{definition} 
 
The sentence $\psi^n_m$ expresses: `\textit{there exist (at least) $n$ limit points of order $m$ that are mutually apart}'.  

The sentence $\rho^n_m$ expresses: `\textit{there exist exactly $n$ limit points of order $m$ that are mutually apart}'. 
 \begin{theorem}\label{T:tauplustau}\begin{enumerate}[\upshape (i)]\item For all $m,n,p,q>0$,\\
 $(n\otimes \mathcal{T}_m,=)\models \psi^q_p$ if and only if either: $p+1< m$ or: $p+1=m$ and  $q\le n $.  \item For all $m,n,p,q>0$,
 $(n\otimes \mathcal{T}_m,=)\models \rho_p^q$ if and only if $p+1=m$ and $n=q$.
 \item For all $m,p,q>0$, $(\omega\otimes \mathcal{T}_m,=)\models \psi^q_p$ if and only if  $p< m$.\end{enumerate} \end{theorem} 
 
 \begin{proof} (i) Note the following:
 
  If $p+1<m$ and $n>0$, then $\mathcal{T}_m$ and also $n\otimes\mathcal{T}_m$ contain infinitely many limit points of order $p$ that are mutually apart. 
 
 If $p+1=m$ and $n>0$, then $n \otimes \mathcal{T}_m$ contains exactly $n$ limit points of order $p$ that are mutually apart: the points $\langle i\rangle\ast \underline 0$, where $i<n$, so  $(n\otimes \mathcal{T}_m,=)\models \psi^q_p$ if and only if $q\le n$.
 
 If $p< m$, then $\omega \times \mathcal{T}_m$ contains infinitely many limit points of order $p$ that are mutually apart.

 \smallskip The proofs of (i), (ii) and (iii) follow easily from these observations.  
  \end{proof}
\begin{definition} $\;$

For each $k$, for each $s$ in $\omega^k$, we define: $\mathcal{T}_s=\bigcup_{i<k}\{\langle i\rangle \ast\delta\mid \delta \in \mathcal{T}_{s(i)}\}$. 

 For each $\alpha$, we define: $\mathcal{T_\alpha}:=\bigcup_i \{\langle i \rangle\ast \delta\mid\delta \in \mathcal{T}_{\alpha(i)}\}$. \end{definition}
 
 \begin{definition} Let $\mathcal{F}_0, \mathcal{F}_1\subseteq\mathcal{N}$ and assume $\varphi:\mathcal{F}_0\rightarrow\mathcal{F}_1$. 
 
  $\varphi$ is a \emph{(surjective)} map from $\mathcal{F}_0$ \emph{onto} $\mathcal{F}_1$ if and only if $\forall \beta \in \mathcal{F}_1\exists \alpha\in \mathcal{F}_0[\varphi|\alpha =\beta]$. 
 
 $\mathcal{F}_0$ \emph{is equivalent to} $\mathcal{F}_1$, notation: $\mathcal{F}_0\sim\mathcal{F}_1$, if and only if there exists $\varphi:\mathcal{F}_0\rightarrow\mathcal{F}_1$ that is both injective\footnote{See Definition \ref{D:functionspreads}.} and surjective.  \end{definition}

\begin{theorem}\label{T:finitesumstoys} \begin{enumerate}[\upshape (i)] \item For each $m$, $\mathcal{T}_m\oplus\mathcal{T}_{m+1}\sim \mathcal{T}_{m+1}$. \item For all $m,n$, if $m<n$, then $\mathcal{T}_m\oplus\mathcal{T}_n\sim\mathcal{T}_{n}$. 
\item For all $k$, for all $s$ in $\omega^k$, there exist $m,n$ such that $\mathcal{T}_s\sim n\otimes \mathcal{T}_m$.  \end{enumerate}\end{theorem} 

\begin{proof} (i) Let $m$ be given. Define $\varphi:\mathcal{T}_m\oplus \mathcal{T}_{m+1}\rightarrow \mathcal{T}_{m+1}$ such that, for all $\delta$ in $\mathcal{T}_m$, $\varphi|\langle 0\rangle \ast\delta = \langle 1\rangle\ast S\circ\delta$, and, for each $\delta$ in $\mathcal{T}_{m+1}$, $\varphi|\langle 1\rangle \ast \delta = \langle 0 \rangle \ast \delta$. Clearly, $\varphi$ is a one-to-one function mapping $\mathcal{T}_m\oplus \mathcal{T}_{m+1}$ onto $\mathcal{T}_{m+1}$.

\smallskip (ii) Let $m$ be given. We use induction on $n$. The case $n=m+1$ has been treated in (i).  Now let $n$ be given such that $m<n$ and $\mathcal{T}_m\oplus\mathcal{T}_n\sim\mathcal{T}_n$. Then $\mathcal{T}_m\oplus \mathcal{T}_{n+1}\sim \mathcal{T}_m\oplus (\mathcal{T}_n\oplus \mathcal{T}_{n+1})\sim(\mathcal{T}_m\oplus \mathcal{T}_n)\oplus \mathcal{T}_{n+1} \sim \mathcal{T}_n\oplus \mathcal{T}_{n+1}\sim \mathcal{T}_{n+1}$. 

\smallskip

(iii) We use induction on $k$. If $s\in \omega^0$, then $s=\langle\;\rangle$ and $\emptyset=\mathcal{T}_s = 0\otimes \mathcal{T}_1$. 

 Now let $k$ be given such that for all $s$ in $\omega^k$ there exist $m,n$ such that $\mathcal{T}_s= n \otimes \mathcal{T}_m$. 
 
 Let $s=t\ast\langle p\rangle$ in $\omega^{k+1}$ be given. Find $m,n$ such that $\mathcal{T}_t=n\otimes \mathcal{T}_m$. Note: $\mathcal{T}_s\sim\mathcal{T}_t\oplus \mathcal{T}_{p}$ and consider several cases. 

\textit{Case (1)}: $t=\langle\;\rangle$. Then $\mathcal{T}_s=1\otimes \mathcal{T}_{p}$.

\textit{Case (2)}: $t \neq \langle\;\rangle$ and  
$p<m$. Then, by (ii): $\mathcal{T}_s \sim \mathcal{T}_t\sim n\otimes \mathcal{T}_m$.

\textit{Case (3)}:  $t \neq \langle\;\rangle$ and  
$p=m$. Then: $\mathcal{T}_s \sim \mathcal{T}_t\oplus \mathcal{T}_m\sim (n+1)\otimes \mathcal{T}_m$.

\textit{Case (4)}:  $t \neq \langle\;\rangle$ and  
$p>m$. Then,  by (ii): $$\mathcal{T}_s \sim\mathcal{T}_t\oplus \mathcal{T}_{p
}\sim\underbrace{\mathcal{T}_m\oplus\ldots\oplus\mathcal{T}_m}_{n}\oplus \mathcal{T}_{p
}\sim\mathcal{T}_p\sim 1\otimes \mathcal{T}_{p}.$$
\end{proof}

\begin{theorem}[$EQ$ has continuum many complete extensions\footnote{Note that there exists an embedding $\rho:\mathcal{C}\hookrightarrow \{\zeta\in[\omega]^\omega\mid\zeta(0)=2\}$.}]\label{T:uncountablymany}$\;$\begin{enumerate}[\upshape (i)] \item For each $\alpha$, $\mathcal{I}(\mathcal{T}_\alpha)=\bigcup_i\{\langle i\rangle \ast\delta\mid \delta \in \mathcal{T}_{\alpha(i)}\;\wedge\;\exists p[\delta(p)+1=\alpha(i)]\}$. \item For all $\alpha$, for all $n$, $(\mathcal{T}_\alpha,=)\models \psi_n$ if and only if $\exists i[\alpha(i)>n]$. \item For all $\alpha$, for all $n$, $(\mathcal{T}_\alpha,=)\models \rho_n$ if and only if \\$\exists i[\alpha(i)=n+1\;\wedge\;\forall j[\alpha(j)=n+1\rightarrow i=j]]$.\item For all $\zeta, \eta$ in $[\omega]^\omega$, if $\zeta\perp\eta$ and $\zeta(0)=\eta(0)=2$,  \\then  there exists a sentence $\psi$ such that $(\mathcal{T}_\zeta,=)\models \psi$ and $(\mathcal{T}_\eta,=)\models \neg\psi$. \end{enumerate} \end{theorem} 

\begin{proof} (i) Use Theorem \ref{T:psirho}(i).

\smallskip (ii) Note that, for each $\alpha$, for each $n$, $(\mathcal{T}_\alpha,=)\models \psi_n$ if and only if $\mathcal{T}_\alpha$ contains a limit point of order $n$  if and only if  $\exists i[\alpha(i)>n]$. 

\smallskip (iii) Note that, for each $\alpha$, for each $n$, $(\mathcal{T}_\alpha,=)\models \rho_n$ if and only if $\mathcal{T}_\alpha$ contains exactly one limit point of order $n$  if and only if  \\$\exists i[\alpha(i)=n+1\;\wedge\;\forall j[\alpha(j)=n+1\rightarrow i=j]]$.

\smallskip (iv) Using (iii), note that, for all $\zeta$ in $[\omega]^\omega$, if $\zeta(0)>1$, then $\forall n[(\mathcal{T}_\zeta,=) \models \rho_n$ if and only if $\exists p[\zeta(p)= n+1]$.

 Conclude that, for all $\zeta,\eta$ in $[\omega]^\omega$, for all $p$, if $\zeta(0)=\eta(0)=2$ and $\zeta \perp\eta$ and \\$p:=\mu i[\zeta(i)\neq \eta(i)]$ and $\zeta(p)<\eta(p)$, then $\neg \exists i[\eta(i)=\zeta(p)]$, and, therefore, \\$(\mathcal{T}_\zeta,=)\models \psi_{\zeta(p)-1}$ and $(\mathcal{T}_\eta,=)\models \neg \psi_{\zeta(p)-1}$.  \end{proof}
\subsection{Finitary spreads suffice}
\begin{definition} Assume $Spr(\beta)$. $\beta$ is called a \emph{finitary spread-law} or a \emph{fan-law} if and only if $\exists \gamma\forall s[\beta(s)=0\rightarrow \forall n[\beta(s\ast\langle n \rangle)=0\rightarrow n\le \gamma(s)]]$. 

$\mathcal{X}\subseteq \mathcal{N}$ is a \emph{fan} if and only if there exists a fan-law $\beta$ such that $\mathcal{X}=\mathcal{F}_\beta$. \end{definition}

Note that the toy spreads $\mathcal{T}_0, \mathcal{T}_1, \ldots$ are fans.

The set $\mathcal{T}_\alpha$, however, is a spread but, in general, not a fan.

\smallskip
Define, for each $\alpha$,  $\mathcal{T}^\ast_\alpha:= \overline{\bigcup_n \underline{\overline 0} n\ast\langle 1\rangle\ast \mathcal{T}_{\alpha(n)}}$.\footnote{For each $\mathcal{X}\subseteq \mathcal{N}$, $\overline{\mathcal{X}}:=\{\alpha\mid\forall n\exists \beta \in \mathcal{X}[\overline \alpha n\sqsubset \beta]\}$ is the \textit{closure} of $\mathcal{X}$. \\$\bigcup_n \underline{\overline 0} n\ast\langle 1\rangle\ast \mathcal{T}_{\alpha(n)}$, in general,  is not a spread, but its closure is.}

Note that, for each $\alpha$, $\mathcal{T}_\alpha^\ast$ is a fan.

One may prove a statement very similar to Theorem \ref{T:uncountablymany}(iv): \begin{quote} \textit{For all $\zeta, \eta$ in $[\omega]^\omega$, if $\zeta\perp\eta$ and $\zeta(0)=\eta(0)=2$, then there exists a sentence $\psi$ such that $(\mathcal{T}^\ast_\zeta, =)\models \psi$ and $(\mathcal{T}^\ast_\eta, =)\models\neg \psi$.} \end{quote}

\subsection{Comparison with an older theorem} The first-order theory $DLO$ of \textit{dense linear orderings without endpoints} is formulated in a first-order language with binary predicate symbols $=$ and $\sqsubset$ and consists of the following axioms: \begin{enumerate} \item $\mathsf{\forall x[x\sqsubset x]}$, \item $\mathsf{\forall x\forall y\forall z[(x\sqsubset y\;\wedge\;y\sqsubset z)\rightarrow x\sqsubset z]}$, \item $\mathsf{\forall x\forall y[
\bigl(\neg (x\sqsubset y)\;\wedge\;\neg (y\sqsubset x)\bigr)\rightarrow x=y]}$. \item $\mathsf{\forall x\forall y[x\sqsubset y\rightarrow\forall z[x\sqsubset z\;\vee\;z\sqsubset y]]}$, \item $\mathsf{\forall x\exists y[x\sqsubset y]\;\wedge\;\forall x\exists y[y\sqsubset x]}$, \item  $\mathsf{\forall x\forall y[x\sqsubset y\rightarrow\exists z[x\sqsubset z\;\wedge\;z\sqsubset y]]}$, and \item axioms of equality. \end{enumerate}

$(\mathcal{R}, =_\mathcal{R}, <_\mathcal{R})$ realizes $DLO$. 

Let $DLO^-$ be the theory one obtains from $DLO$ by leaving out axiom (4). If one defines a relation $<'_\mathcal{R}$ on $\mathcal{R}$ by:
$\forall x\forall y[x<'_\mathcal{R} y\leftrightarrow \neg\neg(x<_\mathcal{R}y)]$, then $(\mathcal{R}, =_\mathcal{R}, <'_\mathcal{R})$ realizes $DLO^-$ but not $DLO$.

In \cite[Theorem 2.4]{veldmanjanssen90} one constructs a function $\alpha\mapsto A_\alpha$ associating to each element $\alpha$ of $2^\omega=\mathcal{C}$ a subset $A_\alpha$ of the set  $\mathcal{R}$ of the real numbers such that, for each $\alpha $ in $\mathcal{C}$,  $A_\alpha$ is dense in $(\mathcal{R}, <_\mathcal{R})$, and,   for all $\alpha, \beta$ in $\mathcal{C}$, if $\alpha\perp \beta$, then there exists a sentence $\psi$ such that $(A_\alpha, <_\mathcal{R})\models \psi$ and $(A_\beta, <_\mathcal{R})\models \neg \psi$. 

Note: each structure $(A_\alpha, <_\mathcal{R})$ realizes $DLO$.
The (intuitionistic) theory $DLO$ thus has continuum many complete extensions.  
\footnote{Classically, $Th\bigl((\mathbb{Q}, <)\bigr)$ is the one and only complete extension of $DLO$.}

Theorem \ref{T:uncountablymany}(iii) strengthens this result. 

\smallskip
One may obtain the result of \ref{T:uncountablymany}(iii) with subsets of $\mathcal{R}$ as well as with subsets of $\mathcal{N}$. 
Define an infinite sequence $\mathcal{U}_0, \mathcal{U}_1, \ldots$ of subsets of $\mathcal{R}$ by: 

$\mathcal{U}_0:=\emptyset$ and $\mathcal{U}_1:=\{0_\mathcal{R}\}$, and for each $m>0$, $\mathcal{U}_{m+1}= \overline{\bigcup_n \frac{1}{2^{n+1}}+\frac{1}{2^{n+2}}\cdot_\mathcal{R} \mathcal{U}_m}.$\footnote{For each $\mathcal{X}\subseteq \mathcal{R}$, $\overline{\mathcal{X}}:=\{x\in \mathcal{R}\mid \forall n\exists y\in \mathcal{X}[|x-y|< \frac{1}{2^n}]\}$ is the \textit{closure} of $\mathcal{X}$.}  

For each $m$, one may define $\varphi:\mathcal{T}_m\rightarrow \mathcal{U}_m$ such that $\varphi$ is surjective and satisfies: $\forall \delta\in\mathcal{T}_m\forall \zeta \in \mathcal{T}_m[\delta\perp\zeta\leftrightarrow \varphi|\delta\;\#_\mathcal{R}\;\varphi|\zeta]$.

It follows that, for each $m$, the structures $(\mathcal{T}_m,=)$ and $(\mathcal{U}_m, =_\mathcal{R})$ are elementarily equivalent.

Define, for each $\alpha$ in $[\omega]^\omega$, $A_\alpha:=\bigcup_n n+_\mathcal{R}\mathcal{U}_{\alpha(n)}$.

Note: for all $\alpha,\beta$ in $[\omega]^\omega$, if $\alpha\perp \beta$, then there exists a sentence $\psi$ such that $(A_\alpha, =_\mathcal{R})\models \psi$ and $(A_\beta, =_\mathcal{R})\models \neg\psi$.

We thus obtain from Theorem \ref{T:uncountablymany} a result similar to \cite[Theorem 2.4]{veldmanjanssen90}, this time using not the ordering relation $<_\mathcal{R}$ but only the equality relation $=_\mathcal{R}$.

Note that the relation $=_\mathcal{R}$ is definable in the structure $(\mathcal{R},<_\mathcal{R})$ as \\$\forall x\in\mathcal{R}\forall y \in \mathcal{R}[x=_\mathcal{R}y\leftrightarrow  \bigl(\neg(x<_\mathcal{R}y)\;\wedge\;\neg(y<_\mathcal{R}x)\bigr)]$.  
\section{The Vitali equivalence relation}\label{S:vitali}

For all $\alpha, \beta$, we define $$\alpha\sim_V\beta\leftrightarrow \exists n\forall m>n[\alpha(m)=\beta(m)].$$

The relation $\sim_V$ will be called the \textit{Vitali equivalence relation}.

This is because  the relation $\sim_V$ on $\mathcal{N}$ resembles the relation $\sim_\mathbb{Q}$ on the set $\mathcal{R}$ of the real numbers defined by: $$x\sim_\mathbb{Q}y\leftrightarrow \exists q\in\mathbb{Q}[x-_\mathcal{R}y=q].$$

The relation $\sim_\mathbb{Q}$ has played an important r\^ole in classical set theory.

If one constructs, using the axiom of choice, within the interval $[0,1]$, a \textit{transversal} for this equivalence relation, that is: a complete set of mutually inequivalent representatives,  one obtains a set that is not Lebesgue measurable. This discovery is due to G.~Vitali.

Note: $(\mathcal{N}, \sim_V)\models EQ$. 

The following theorem brings to light an important difference between $(\mathcal{N},=)$ and $(\mathcal{N}, \sim_V)$. 

\begin{definition}\label{D:stable} A proposition $P$ is  \emph{stable} if and only if $\neg\neg P\rightarrow P$.

A binary relation $\sim$ on $\mathcal{N}$ is  \textit{stable} if and and only if \\$\forall\alpha\forall\beta[\neg\neg(\alpha\sim\beta)\rightarrow \alpha\sim\beta]$\footnote{The term `\textit{stable}'  has been introduced by D. Van Dantzig, who hoped to be able to reconstruct `classical', non-intuitionistic mathematics within the stable part of intuitionistic mathematics, see \cite{dantzig}.}.\end{definition} 
\begin{theorem}[Equality is stable but the Vitali equivalence relation is not stable]\label{T:vitaliunstable}

 \begin{enumerate}[\upshape (i)] $\;$ \indent \item $(\mathcal{N}, =)\models \mathsf{\forall x \forall y[\neg\neg(x=y)\rightarrow x=y]}$.
\item $(\mathcal{N}, \sim_V)\models \mathsf{\forall x\neg \forall y[\neg\neg(x=y)\rightarrow x=y]}$. \end{enumerate} \end{theorem}

\begin{proof} (i) Note: for all $\alpha, \beta$, $\alpha=\beta\leftrightarrow \neg (\alpha\;\#\;\beta)$, and, therefore: \\$\neg\neg(\alpha=\beta)\leftrightarrow \neg \neg\neg(\alpha\;\#\;\beta)\leftrightarrow \neg(\alpha\;\#\; \beta)\leftrightarrow \alpha=\beta$.

\smallskip
(ii) Let $\gamma$ be given. 

Consider $\mathcal{F}^\gamma:=\{\alpha\mid\forall m\forall n[\bigl(\alpha(m)\neq\gamma(m)\;\wedge\;\alpha(n)\neq \gamma(n)\bigr)\rightarrow m=n]$. 

$\mathcal{F}^\gamma$ is the set of all $\alpha$ that differ at at most one place from $\gamma$.

Note that $\mathcal{F}^\gamma$ is a spread.

We have two claims. 

\textit{First claim:} $\forall \alpha \in \mathcal{F}^\gamma[\neg\neg(\alpha\sim_V \gamma]$.

The proof is as follows. Let $\alpha$ in $\mathcal{F}^\gamma$ be given. Distinguish two cases.

\smallskip
\textit{Case (1)}. $\exists n[\alpha(n)\neq \gamma(n)]$. Find $n$ such that $\alpha(n)\neq \gamma(n)$ and conclude: \\$\forall m>n[\alpha(m)=\gamma(m)]$ and $\alpha\sim_V\gamma$.

\smallskip
\textit{Case (2)}. $\neg\exists n[\alpha(n)\neq \gamma(n)]$. Conclude: $\forall n[\alpha(n)=\gamma(n)]$ and $\alpha\sim_V\gamma$. 

\smallskip
We thus see: if $\exists n[\alpha(n)\neq \gamma(n)]\;\vee\;\neg\exists n[\alpha(n)\neq \gamma(n)]$, then $\alpha\sim_V\underline \gamma$. 

As $\neg\neg(\exists n[\alpha(n)\neq \gamma(n)]\;\vee\;\neg\exists n[\alpha(n)\neq \gamma(n)])$, also $\neg\neg(\alpha\sim_V\gamma)$. 

\smallskip \textit{Second claim:} $\neg\forall \alpha \in \mathcal{F}^\gamma[\alpha \sim \gamma]$.

In order to see this, assume: $\forall \alpha \in \mathcal{F}^\gamma[\alpha \sim \gamma]$, that is:\\ $\forall \alpha \in \mathcal{F}\exists n\forall m>n[\alpha(m)=\gamma(m)]$.
Using Lemma \ref{L:bcpspreads}, find $p,n$ such that \\$\forall \alpha \in \mathcal{F}^\gamma[\overline{\gamma}p\sqsubset \alpha\rightarrow \forall m>n[\alpha(m)=\gamma(m)]]$. 
Define $m:=\max(p,n+1)$ and define $\alpha$ such that  $\forall n[\alpha(n)\neq\gamma(n)\leftrightarrow n=m]$. 
Note: $\overline{\gamma} p\sqsubset \alpha$ and $m>n$ and  $\alpha(m)\neq \gamma(m)$. Contradiction.
 
 \smallskip Combining our two claims, we see:
\\ not: for all $\alpha$, if $\neg\neg(\alpha\sim_V \gamma)$ then $\alpha\sim_V \gamma$.  

Conclude: $(\mathcal{N}, \sim_V)\models \mathsf{\forall x\neg\forall y[\neg\neg(x=y)\rightarrow x=y]}$.
 \end{proof}

 It follows from Theorem \ref{T:vitaliunstable} that there is no relation $\#_V$ on $\mathcal{N}$ satisfying the requirements of an apartness relation\footnote{See \cite[p. 256]{troelstra}} with respect to $\sim_V$:
 
 \begin{enumerate}[\upshape (i)] \item $\forall \alpha\forall \beta[\neg (\alpha\;\#_V\;\beta)\leftrightarrow \alpha \sim_V\beta]$ \item $\forall\alpha\forall\beta[\alpha\;\#_V\;\beta\rightarrow \beta\;\#_V\;\alpha]$ \item $\forall \alpha\forall \beta[\alpha\;\#_V\;\beta \rightarrow\forall \gamma[\alpha\;\#_V\;\gamma\;\vee\;\gamma\;\#_V\;\beta]]$. \end{enumerate}
 
 The existence of an apartness $\#_V$ would imply, by the first one of these requirements, that $\sim_V$ is a stable relation, as, for any proposition $P$, $\neg\neg\neg P\leftrightarrow\neg P$.
 
 The next Theorem now is no surprise:
\begin{theorem}\label{T:vitalinoapartness}  $(\mathcal{N}, \sim_V)\models \mathsf{\forall x \forall y[\neg}AP(\mathsf{x,y})]$. \end{theorem}

\begin{proof} Let $\alpha, \beta$ be given. 

Assume $(\mathcal{N},\sim_V)\models AP[\alpha, \beta]$, that is, $\forall \gamma[\gamma\nsim_V \alpha\;\vee\;\gamma\nsim_V \beta]$. 

Applying Lemma \ref{L:bcpdisj}, find $p$ such that \\either $\forall \gamma [\overline{\alpha}p\sqsubset \gamma \rightarrow \gamma \nsim_V \alpha]$ or  $\forall \gamma [\overline{\alpha}p\sqsubset \gamma \rightarrow \gamma \nsim_V \beta]$. 

The first of these two alternatives is wrong, as $\overline \alpha p\sqsubset \alpha \;\wedge\; \alpha \sim_V \alpha$.

Conclude:  $\forall \gamma [\overline{\alpha}p\sqsubset \gamma \rightarrow \gamma \nsim_V \beta]$. 

Define $\gamma$ such that $\overline\alpha p\sqsubset \gamma $ and $\forall i>p[\gamma(i)=\beta(i)]$.

Note: $\overline \alpha p\sqsubset \gamma \;\wedge\; \gamma \sim_V \beta $.

Contradiction.

Conclude: $(\mathcal{N}, =_V)\models \neg AP[\alpha, \beta]$.

We thus see: $(\mathcal{N}, =_V)\models \mathsf{\forall x \forall y[\neg}AP(\mathsf{x,y})]$.
\end{proof}

Clearly, the relation defined by the formula $AP$ in the structure $(\mathcal{N}, \sim_V)$ fails to satisfy the first requirement for an apartness relation with respect to $\sim_V$. 

It follows from Theorem \ref{T:vitalinoapartness} that $(\mathcal{N}, \sim_V)$, while realizing $T_{inf}$, does not realize $T^+_{inf}$, see Definitions \ref{D:Tinf} and \ref{D:Tinfplus}. 

\section{A first Vitali variation}\label{S:vitalivar}
There are many intuitionistic versions of the classical Vitali equivalence relation. This is obvious to someone who knows that there are many variations upon the notion of a finite and decidable subset of $\mathbb{N}$, see \cite{veldman1995} and \cite[Section 3]{veldman2005}.
\begin{definition}

We define an infinite sequence 
$\sim^0_V, \sim^1_V, \ldots$ of relations on $\mathcal{N}$ such that $\sim^0_V\;=\;\sim_V$ and, for each $i$,
 
 $$\alpha\sim_V^{i+1}\beta\leftrightarrow  \exists n\forall m>n[\alpha(m)\neq\beta(m)\rightarrow \alpha\sim^i_V\beta].$$
 We also define:
 $$\alpha\sim^\omega_V\beta\leftrightarrow \exists i[\alpha\sim^i_V\beta].$$ \end{definition}
 
 \begin{theorem}\label{T:vitalivariations}$\;$\begin{enumerate}[\upshape (i)]\item $\forall i \forall n\forall s \in \omega^n\forall t \in \omega^n\forall \alpha\forall \beta[s\ast \alpha\sim^i_V t\ast\beta \leftrightarrow \alpha\sim_V^i \beta]$. \item $\forall i\forall \alpha\forall \beta [\alpha\sim_V^{i}\beta\rightarrow \alpha \sim_V^{i+1}\beta]$. \item $\forall i\forall\gamma\neg\forall \alpha[\alpha\sim_V^{i+1}\gamma\rightarrow \alpha \sim_V^i\gamma]$. 
 \item $\forall i\forall j\forall \alpha\forall \beta\forall \gamma[(\alpha \sim_V^i \beta\;\wedge\;\beta \sim_V^j \gamma)\rightarrow \alpha\sim_V^{i+j}\gamma]$. \item $\sim_V^\omega$ is an equivalence relation on $\mathcal{N}$. \end{enumerate}\end{theorem} 
 
 \begin{proof} (i) One proves this easily by induction.
 
 \smallskip
 (ii) Obvious. 
 
 \smallskip (iii) Let $\gamma$ be given. 
 
 For each $i$, define $\mathcal{F}^\gamma_i:=\{\alpha\mid \forall s \in [\omega]^{i+1}\exists j<i+1[\alpha\circ s(j)=\gamma\circ s(j)]\}$. 
 
 Note: for each $i$, $\mathcal{F}^\gamma_i$ is a spread, and $\mathcal{F}^\gamma_i \subsetneq \mathcal{F}^\gamma_{i+1}$.
 
 For each $i$, $\mathcal{F}^\gamma_i$ consists of all $\alpha$ that assume at most $i$ times a value different from the value assumed by $\gamma$. In particular, $\mathcal{F}^0_\gamma=\{\gamma\}$.

 Note: for all $i, m, \alpha, \beta$, \\if $m=\mu n[\alpha(n)\neq\gamma(n)]$ and  $\alpha =\overline{\alpha}(m+1)\ast\beta$, then $\alpha \in \mathcal{F}^\gamma_{i+1}\leftrightarrow \beta \in \mathcal{F}^\gamma_i$. 
 
 \smallskip
 We have two claims. 
 
 \smallskip \textit{First claim:} $\forall i\forall \alpha\in \mathcal{F}^\gamma_i[\alpha\sim_V^i\gamma]$. 
 
 We prove this claim by induction.
 
 The starting point of the induction is the  observation: \\$\forall \alpha \in \mathcal{F}_0^\gamma[\alpha =\gamma]$, so $\forall \alpha \in \mathcal{F}^\gamma_0[\alpha \sim_V^0\gamma]$.
 
 \smallskip
 Now assume $i$  is given such that $\forall \alpha \in \mathcal{F}_i^\gamma[\alpha \sim^i_V\gamma]$. 
 
 Assume $\alpha\in\mathcal{F}^\gamma_{i+1}$ and $\exists n[\alpha(n)\neq \gamma(n)]$. Find $n$ such that $\alpha(n)\neq \gamma(n)$. Find $\beta$ such that $\alpha=\overline\alpha(n+1)\ast\beta$, and note: $\beta \in \mathcal{F}^\gamma_i$ and thus, by the induction hypothesis, $\beta\sim_V^i \gamma$. Conclude, using (i): $\alpha \sim_V^i \gamma$. 
 
 We thus see:\\ $\forall \alpha \in \mathcal{F}^\gamma_{i+1}[\exists n[\alpha(n)\neq \gamma(n)]\rightarrow \alpha \sim^i_V\gamma]$, that is: $\forall \alpha \in \mathcal{F}^\gamma_{i+1}[\alpha\sim^{i+1}_V \gamma]$. 
 
 This completes the proof of the induction step.

\smallskip \textit{Second claim:} $\forall i \neg \forall \alpha \in \mathcal{F}^\gamma_{i+1}[\alpha\sim^i_V \gamma]$.

We again use induction.

We first prove: $\neg \forall \alpha \in \mathcal{F}^\gamma_1[\alpha\sim_V \gamma]$.

Assume $\forall \alpha \in \mathcal{F}^\gamma_1[\alpha\sim_V \gamma]$, that is: $\forall \alpha \in \mathcal{F}_1\exists n\forall m>n[\alpha(m) =\gamma(m)]$. 

Note: $\gamma \in \mathcal{F}_1^\gamma$ and $\mathcal{F}^1_\gamma$ is a spread. 

Using Lemma \ref{L:bcpspreads},  find $p, n$ such that $\forall \alpha \in \mathcal{F}_1[\overline \gamma p\sqsubset\alpha\rightarrow \forall m>n[\alpha(m)=\gamma(m)]$. Define $m:=\max(n+1,p)$ and define $\alpha$ such that  $\forall n[\alpha(n)=\gamma(n)\leftrightarrow n\neq m]$.  Note: $\alpha \in \mathcal{F}_1$ and  $\overline \gamma p\sqsubset \alpha$ and $\alpha(m)\neq \gamma(m)$ and $m>n$. Contradiction. 

Conclude:  $\neg \forall \alpha \in \mathcal{F}^\gamma_1[\alpha\sim_V \gamma]$.

\smallskip Now let $i$ be given such that $ \neg \forall \alpha \in \mathcal{F}^\gamma_{i+1}[\alpha\sim_V^i \gamma]$.

 We want to prove: $ \neg \forall \alpha \in \mathcal{F}^\gamma_{i+2}[\alpha\sim_V^{i+1} \gamma]$.
 
 Assume: $\forall \alpha \in \mathcal{F}^\gamma_{i+2}[\alpha\sim_V^{i+1} \gamma]$, that is:\\ $\forall \alpha \in \mathcal{F}^\gamma_{i+2}\exists n\forall m>n[\alpha(m)\neq \gamma(m)\rightarrow \alpha\sim^{i+1}_V\gamma]$. Using Lemma \ref{L:bcpspreads}, find $p,n$ such that $\forall \alpha \in \mathcal{F}^\gamma_{i+2}[(\overline \gamma p\sqsubset \alpha \;\wedge\;m>n\;\wedge\;\alpha(m)\neq \gamma(m))\rightarrow \alpha\sim^i_V\gamma]$. Define $m:=\max(n+1, p)$.
  Let $\beta$ in $\mathcal{F}^\gamma_{i+1}$ be given. Define $\alpha$ such that $m=\mu n[\alpha(n)\neq \gamma(n)]$ and $\forall n> m[\alpha(n)=\beta(n)]$.   Note:  $\alpha \in \mathcal{F}^\gamma_{i+2}$ and $\alpha(m)\neq \gamma(m)$ and $m>n$, so $\alpha \sim_V^{i}\gamma$, and, therefore, by (i), $\beta\sim_V^{i}\gamma$. We thus see: $\forall \beta \in \mathcal{F}^\gamma_{i+1}[\beta\sim_V^i\gamma]$ and, by the induction hypothesis, obtain a contradiction. 
 
 This completes the proof of the induction step.
 
 \smallskip Taking our first and second claim together, we obtain the conclusion: \\$\forall\gamma\forall i\neg \forall \alpha[\alpha\sim^{i+1}_V \gamma\rightarrow \alpha\sim^{i}_V \gamma]$. 
 
 \smallskip (iv) We have to prove:
 
  for all $i$, for all $j$,  $\forall \alpha\forall \beta\forall \gamma[(\alpha \sim_V^i \beta\;\wedge\;\beta \sim_V^j \gamma)\rightarrow \alpha\sim_V^{i+j}\gamma]$.
 
 We use induction on $i+j$ and distinguish four cases.
 
 \smallskip
 \textit{Case (1):} $i=j=0$. Assume $\alpha\sim^0_V\beta$ and $\beta\sim^0_V\gamma$. Find $n,p$ such that $\forall m>n[\alpha(m)=\beta(m)]$ and $\forall m>p[\beta(m)=\gamma(m)$. Define $q:=\max(n, p)$ and note: $\forall m>q[\alpha(m)=\gamma(m)]$. Conclude: $\alpha\sim^0_V\gamma$.
 
 \smallskip
 \textit{Case (2):} $i=0$ and $j>0$.  Assume $\alpha\sim^0_V\beta$ and $\beta\sim^j_V\gamma$. Find $n,p$ such that $\forall m>n[\alpha(m)=\beta(m)]$ and $\forall m>p[\beta(m)\neq\gamma(m)\rightarrow \beta\sim^{j-1}_V\gamma]$. Define $q:=\max(n,p)$.
 
  Assume $m>q$ and note: if $\alpha(m)\neq \gamma(m)$, then $\beta(m) \neq \gamma(m)$ and $\beta\sim^{j-1}_V\gamma$. Using the induction hypothesis, conclude: $\alpha\sim^{j-1}_V \gamma$. 
 
 We thus see: $\forall m>q[\alpha(m)\neq \gamma(m)\rightarrow \alpha\sim^{j-1}_V\gamma]$, that is: $\alpha\sim^j_V\gamma$.

 \smallskip \textit{Case (3):} $i>0$ and $j=0$.  Assume $\alpha\sim^i_V\beta$ and $\beta\sim^0_V\gamma$. Find $n,p$ such that $\forall m>n[\alpha(m)\neq\beta(m)\rightarrow \alpha\sim^{i-1}_V \beta]$ and $\forall m>p[\beta(m)=\gamma(m)]$. Define $q:=\max(n,p)$.
 
  Assume $m>q$ and note: if $\alpha(m)\neq \gamma(m)$, then $\alpha(m) \neq \beta(m)$ and $\alpha\sim^{i-1}_V\beta$. Using the induction hypothesis, conclude: $\alpha\sim^{i-1}_V \gamma$. 
 
 We thus see: $\forall m>q[\alpha(m)\neq \gamma(m)\rightarrow \alpha\sim^{i-1}_V\gamma]$, that is: $\alpha\sim^i_V\gamma$. 
 
 \smallskip \textit{Case (4):} $i>0$ and $j>0$. 
 Assume $\alpha\sim^i_V\beta$ and $\beta\sim^j_V\gamma$. Find $n,p$ such that $\forall m>n[\alpha(m)\neq\beta(m)\rightarrow \alpha\sim^{i-1}_V \beta]$ and $\forall m>p[\beta(m)\neq\gamma(m)\rightarrow \beta\sim^{j-1}_V\gamma]$. Define $q:=\max(n,p)$.
 
 Assume $m>q$ and $\alpha(m)\neq \gamma(m)$. Then \textit{either:} $\alpha(m)\neq\beta(m)$ and $\alpha\sim^{i-1}\beta$, and, by the induction hypothesis, $\alpha\sim^{i+j-1}_V\gamma$, \textit{or:} $\beta(m)\neq\gamma(m)$ and $\beta\sim^{j-1}\gamma$ and, by the induction hypothesis, $\alpha\sim^{i+j-1}\gamma$.
 
 We thus see: $\forall m>q[\alpha(m)\neq\gamma(m)\rightarrow \alpha\sim^{i+j-1}\gamma]$. 
 Conclude: $\alpha\sim^{i+j}\gamma$. 
 
 \smallskip (v) is an easy consequence of (iv). 
 \end{proof}

  The next Theorem shows that the structures $(\mathcal{N}, \sim_V)$ and $(\mathcal{N}, \sim^\omega_V)$ have a property in common.
 
 \begin{theorem}[$\sim^\omega_V$ is not stable]\label{T:omegaunstable} $\;$ 
 
  $(\mathcal{N}, \sim_V^\omega)\models \mathsf{\forall x\neg\forall y[\neg\neg(x=y)\rightarrow x =y]}$. \end{theorem}
  
  \begin{proof} Let $\gamma$ be given. 
  
  We repeat a definition we gave in the proof of Theorem \ref{T:vitalivariations}(iii).
  
  For each $i$, $\mathcal{F}^\gamma_i:=\{\alpha\mid\forall s \in [\omega]^{i+1}\exists j<i+1[\alpha\circ s(j)=\gamma\circ s(j)]\}$.

  In the proof of Theorem \ref{T:vitalivariations}(iii), we saw: $\forall i\forall \alpha \in \mathcal{F}^\gamma_i[\alpha\sim^i_V \gamma]$.
  
  Conclude: $\forall i\forall \alpha \in \mathcal{F}^\gamma_i[\alpha\sim^\omega_V \gamma]$.
  
  We now define: $\mathcal{F}^\gamma_\omega:=\{\alpha\mid \forall i[i=\mu n[\alpha(n)\neq \gamma(n)]\rightarrow\alpha\in \mathcal{F}_{i+1}]\}$.  
  
  Like each $\mathcal{F}^\gamma_i$,  $\mathcal{F}^\gamma_\omega$ is a spread, and $\gamma \in \mathcal{F}^\gamma_\omega$. 
  
  We have two claims.
  
  \textit{First claim:} $\forall \alpha \in \mathcal{F}^\gamma_\omega[\neg\neg(\alpha\sim_V^\omega \gamma)]$. 
  
  The argument
   is as follows.  Let $\alpha$ in $\mathcal{F}^\gamma_\omega$ be given and distinguish two cases.
   
   \textit{Case (1):} $\neg\exists n[\alpha(n)\neq \gamma(n)]$.  Then $\alpha=\gamma$ and $\alpha\sim_V^\omega \gamma$.
   
   \textit{Case (2):}  $\exists n[\alpha(n)\neq \gamma(n)]$. Find $i:=\mu n[\alpha(n)\neq\gamma(n)]$. Note:  $\alpha \in \mathcal{F}_{i+1}^\gamma$ and $\alpha\sim^\omega_V \gamma$. \\As $\neg\neg (\exists n[\alpha(n)\neq \gamma(n)]\;\vee\;\neg\exists n[\alpha(n)\neq \gamma(n)])$, also $\neg\neg(\alpha\sim_V^\omega \gamma)$. 
   
   \textit{Second claim:} $\neg \forall \alpha \in \mathcal{F}^\gamma_\omega[\alpha\sim_V^\omega \gamma]$. 
   
   In order to see this, assume: $\forall \alpha \in \mathcal{F}^\gamma_\omega[\alpha\sim_V^\omega \gamma]$, that is: $\forall \alpha \in \mathcal{F}_\omega\exists i[\alpha\sim^i_V \gamma]$. 
   
   Using Lemma \ref{L:bcpspreads}, find $p,i$ such that $\forall \alpha \in \mathcal{F}^\gamma_\omega[\overline \gamma p\sqsubset \alpha\rightarrow \alpha\sim^i_V \gamma]$.

   Define $q:=\max(p, i+1)$. Let $\alpha$ in $\mathcal{F}^\gamma_q$ be given. 
    Define $\beta$ such that \\$\forall n<q[\beta(n)=\gamma(n)]$ and $\beta(q)\neq \gamma(q)$ and $\forall n>q[\beta(n)=\alpha(n)]$.
   
     Note: $\beta\in \mathcal{F}_{q+1}$ and $q=\mu n[\beta(n)\neq\gamma(n)]$, and, therefore, $\beta\in \mathcal{F}_\omega^\gamma$.

     As $\overline \gamma q\sqsubset \beta$, we conclude: $\beta \sim_V^i\gamma$. 
     
     As $\beta\sim^0_V\alpha$, also $\alpha\sim_V^i\gamma$. 
     
     We thus see: $\forall \alpha \in \mathcal{F}_q[\alpha\sim_V^i \gamma]$.
     
     As $q>i$, this contradicts the Second claim in the proof of Theorem \ref{T:vitalivariations}(iii).

   Taking our two claims together, we  conclude: 
   
$\forall \gamma\neg \forall\alpha\in\mathcal{F}^\gamma_\omega[\neg\neg(\alpha\sim^\omega_V\gamma)\rightarrow\alpha\sim^\omega_V\gamma]$. 
 
 Conclude: 
 $(\mathcal{N}, \sim^\omega_V)\models \mathsf{\forall x\neg\forall y[\neg\neg(x=y)\rightarrow x =y]}$.
   \end{proof}
  
 We did not succeed in finding a sentence $\psi$ such that $(\mathcal{N}, \sim_V)\models \psi$ and \\$(\mathcal{N}, \sim^\omega_V)\models \neg \psi$.
  
  \section{More and more Vitali relations}\label{S:morevitali}
  
  In \cite{veldman1995}, \cite{veldman1999} and \cite[Section 3]{veldman2005}, one studies the set $$\mathbf{Fin}:=\{\alpha\mid\alpha\sim_V\underline 0\}=\{\alpha\mid\exists n \forall m>n[\alpha(m)=0]\}.$$ For each $\alpha$, $\alpha\in\mathbf{Fin}$ if and only if $D_\alpha:=\{m\mid\alpha(m)\neq 0\}$ is a \textit{finite} subset of $\mathbb{N}$.

  For each $i$, the set $\{\alpha\mid\alpha\sim_V^i \underline 0\}$ is called, in   \cite{veldman1999} and \cite{veldman2005}, the $i$-th \textit{perhapsive extension} of the set $\mathbf{Fin}$. 
 It is shown, in \cite{veldman1995}, \cite{veldman1999} and \cite{veldman2005},  that the process of building perhapsive extensions of $\mathbf{Fin}$ can be carried on into the transfinite.

  In a similar way, the Vitali equivalence relation  $\sim_V$ admits of transfinitely many extensions.
  
  The relation $\sim_V^\omega$ is only a \textit{first} extension of $\sim_V$. Let us consider a second one. 
  
 \medskip Recall: $\forall\alpha\forall\beta[\alpha\sim_V^\omega\beta\sim\exists i[\alpha\sim_V^ i\beta]]$.
  
  \begin{definition}

We define an infinite sequence 
$\sim_V^{\omega+0}=\sim_V^\omega, \sim^{\omega+1}_V, \sim^{\omega+2}_V, \ldots$ of relations on $\mathcal{N}$, such that,
for each $i>0$,
 $$\alpha\sim_V^{\omega+i+1}\beta\leftrightarrow  \exists n\forall m>n[\alpha(m)\neq\beta(m)\rightarrow \alpha\sim^{\omega+i}_V\beta].$$
 We also define:
 $$\alpha\sim^{\omega+\omega}_V\beta\leftrightarrow \exists i[\alpha\sim^{\omega+i}_V\beta].$$ \end{definition}
 One may prove  analogues of Theorems  \ref{T:vitalivariations} and \ref{T:omegaunstable} and conclude:
 
 $\sim_V^{\omega+\omega}$ is an  equivalence relation on $\mathcal{N}$, properly extending $\sim_V^\omega$, that, like $\sim_V$ and $\sim_V^\omega$, is not stable in the sense of Theorem \ref{T:omegaunstable}.
 
 \smallskip
 
 One may continue and define $\sim_V^{\omega+\omega+\omega}$, and $\sim_V^{\omega+\omega+\omega+\omega}$ and so on. 
 
 The process of building such extensions leads  further into the transfinite, as follows. 
 
 \begin{definition}\label{D:perhapsiveextvitali} Let $R$ be  binary relation  on $\mathcal{N}$. 
 
 We define a binary relation $R^+$ on $\mathcal{N}$ by:
 $$\alpha R^+\beta\leftrightarrow \exists n\forall m>n[\alpha(m)\neq \beta(m)\rightarrow \alpha R \beta].$$

 We  let $\mathcal{E}$ be the least class of binary relations on $\mathcal{N}$ such that \begin{enumerate}[\upshape (i)] \item the Vitali equivalence relation  $\sim_V$ belongs to $\mathcal{E}$, and, \item for every $R$ in $\mathcal{E}$, also $R^+  \in \mathcal{E}$, and, \item for every infinite  sequence $R_0, R_1,  \ldots$ of elements of $\mathcal{E}$, also $\bigcup_i R_i \in \mathcal{E}$. \end{enumerate}
  
  The elements of $\mathcal{E}$ are the \emph{extensions of the  Vitali  equivalence relation}. 
  
   \end{definition}
   
   Note that $<_V^\omega$ and $<_V^{\omega+\omega}$ are in $\mathcal{E}$.
   
   In general, a relation $R$ in $\mathcal{E}$ is not transitive. One may prove, for instance, that the relation $<_V^1$, while belonging to $\mathcal{E}$,  is not transitive.
   
   The next Theorem shows that $\mathcal{E}$ contains many transitive relations.
   
   \begin{theorem}[$\mathcal{E}$ contains many transitive relations]\label{T:transitive} \begin{enumerate}[\upshape (i)] \item $\sim_V$ is transitive.
   
   \item Given any transitive $R$ in $\mathcal{E}$, there exists a transitive $T$ in $\mathcal{E}$ such that $R^+\subseteq T$.
   
   \item Given any infinite and increasing sequence $R_0\subseteq R_1\subseteq \ldots$ of transitive relations in $\mathcal{E}$, also $\bigcup_i R_i$ is a transitive relation in $\mathcal{E}$.\end{enumerate} \end{theorem}
   
    \begin{proof} (i) Obvious.
    
    \smallskip (ii)  We take our inspiration from Theorem \ref{T:vitalivariations} (iv) and (v).

    Let a transitive $R$ in $\mathcal{E}$ be given.
    
    Define an infinite sequence $R^0, R^1, \ldots$ of elements of $\mathcal{E}$ such that $R^0=R$ and, for each $i$, $R^{i+1}= (R^ i)^+$. 
    
    One may prove: for all $i$, for all $j$, $\forall\alpha\forall\beta\forall \gamma[(\alpha R^i\beta\;\wedge\;\beta R^i\gamma)\rightarrow \alpha R^{i+j} \gamma]$, as it is done for the special case $R=\sim_V$ in the proof of Theorem \ref{T:vitalivariations}(iv). 
    
    Define $T:=\bigcup_i R^i$ and note: $T\in \mathcal{E}$, $R^+\subseteq T$ and $T$ is transitive. 
    
   \smallskip (iii) Note: for every increasing sequence $R_0\subseteq R_1\subseteq \ldots$ of transitive relations on $\mathcal{N}$, also $\bigcup_i R_i$ is transitive. 
    \end{proof}
    
    Theorem \ref{T:transitive} will gain significance after Corollary \ref{C:hierarchy}, which shows that, for every $R$ in $\mathcal{E}$, $R\subseteq R^+$ and $\neg(R^+\subseteq R)$. 
    
    We did not succeed in proving that every $R$ in $\mathcal{E}$ extends to a transitive $T$ in $\mathcal{E}$. 
   
   \begin{definition}\label{D:shiftinv}  A binary relation $R$ on $\mathcal{N}$ is \emph{shift-invariant} if and only if \\$\forall \alpha\forall \beta[\alpha R \beta \leftrightarrow (\alpha\circ S) R (\beta\circ S)]$.

   \end{definition}
   
   \begin{lemma}\label{L:shiftinvariant} Every $R$ in $\mathcal{E}$ is shift-invariant. \end{lemma}
   
   \begin{proof} The proof is a straightforward exercise in induction on $\mathcal{E}$. Note:
   
   (I) $\sim_V$ is shift-invariant.
   
   (II) For every binary relation $R$ on $\mathcal{N}$, if $R$ is shift-invariant, then $R^+$ is shift-invariant.
   
   (III) For every infinite sequence $R_0, R_1, \ldots$ of binary relations on $\mathcal{N}$, if each $R_n$ is shift-invariant, then $\bigcup_iR_i$ is shift-invariant.
   
   Conclude: every $R$ in $\mathcal{E}$ is shift-invariant. \end{proof}
   \begin{definition}\label{D:Estar} We let $\mathcal{E}^\ast$ be the least class of binary relations on $\mathcal{N}$ such that \begin{enumerate}[\upshape (i)] \item the Vitali equivalence relation $\sim_V$ belongs to $\mathcal{E}^\ast$, and \item for every infinite sequence $R_0, R_1, \ldots $ of elements of $\mathcal{E}^\ast$, also $(\bigcup_i R_i)^+\in \mathcal{E}^\ast.$\end{enumerate}\end{definition}
   
   \begin{lemma}\label{L:eestar} $\mathcal{E}^\ast\subseteq \mathcal{E}$ and, for all $R$ in $\mathcal{E}$, there exists $T$ in $\mathcal{E}^\ast$ such that $R\subseteq T$. \end{lemma}
   \begin{proof} The proofs of the two statements are straightforward, by induction on $\mathcal{E}^\ast$ and $\mathcal{E}$, respectively. \end{proof}
   \begin{theorem}\label{T:hierarchy} For each $R$ in $\mathcal{E}^\ast$, $R\subseteq R^+$ and $\neg(R^+\subseteq R)$. \end{theorem}
   
   \begin{proof}

   \medskip For each $R$ in $\mathcal{E}$, we define $Fin_R:=\{\alpha\mid \alpha R\underline 0\}$.
   \footnote{In \cite{veldman1995},  $\mathcal{X}\subseteq \mathcal{N}$ is called  a \textit{notion of finiteness} if $\mathbf{Fin}\subseteq \mathcal{X}\subseteq \mathbf{Fin}^{\neg\neg}$. For every $R$ in $\mathcal{E}$, $Fin_R$ is a notion a finiteness.}
  
  We prove  for each $R$ in $\mathcal{E}^\ast$  there exists a fan $\mathcal{F}$ such that $\mathcal{F}\subseteq Fin_{R^+}$ and $\neg(\mathcal{F}\subseteq Fin_R)$.
  
  We do so by induction on $\mathcal{E}^\ast$. 
  
  \smallskip (I) Define $\mathcal{F}:=\{\alpha\mid\forall m\forall n[(\alpha(m)\neq 0 \;\wedge\;\alpha(n)\neq 0)\rightarrow m=n]\}$. 
  
  Note that $\mathcal{F}$ is a fan. 
  
  \smallskip
  For each $\alpha$ in $\mathcal{F}$, for each $n$, if $\alpha(n)\neq 0$ then: $\forall m>n[\alpha(m)=0]$ and $\alpha \in Fin_{\sim_V}$. Conclude: for each $\alpha \in \mathcal{F}$, if $\exists n[\alpha(n)\neq 0]$, then $\alpha\in Fin_{\sim_V}$, that is: $\alpha \in Fin_{(\sim_V)^+}$. Conclude: $\mathcal{F}\subseteq Fin_{(\sim_V)^+}$. 
  
  \smallskip Now assume $\mathcal{F}\subseteq Fin_{\sim_V}$, that is: $\forall \alpha \in \mathcal{F}\exists n\forall m>n[\alpha(m)=0]$.  Using Lemma \ref{L:bcpspreads}, find $p,n$ such that $\forall \alpha \in \mathcal{F}[\underline{\overline 0}p\sqsubset \alpha\rightarrow \forall m>n[\alpha(m)=0]]$. 
  \\Define $q:=\max(p, n+1)$ and consider $\alpha:=\underline{\overline 0}q\ast\langle 1 \rangle\ast \underline 0$.
   Contradiction.  
  
  Conclude: $\neg(\mathcal{F}\subseteq Fin_{\sim_V})$.

 \medskip (II) Let $R_0, R_1, \ldots$ be an infinite sequence of elements of $\mathcal{E}$.  
  
  Let $\mathcal{F}_0, \mathcal{F}_1,\ldots $ be an infinite sequence of fans such that,\\ for each $n$, $\mathcal{F}_n\subseteq Fin_{(R_n)^+}$ and $\neg(\mathcal{F}_n\subseteq Fin_{ R_n})$. 
  
  Consider $R:=(\bigcup_i R_i)^+$. 
  
  Define $\mathcal{F}:=\{\alpha\mid \forall n[n=\mu i[\alpha(i)\neq 0]\rightarrow \exists \beta \in \mathcal{F}_{n'}[\alpha=\overline\alpha(n+1)\ast \beta]\}$.\footnote{For each $n$, $n=(n',n'')$, see Section \ref{S:notations}.}
  
  Note that $\mathcal{F}$ is a fan. 
  
  We now prove: $\mathcal{F}\subseteq Fin_{R^+}$ and $\neg(\mathcal{F}\subseteq Fin_R)$. 
  
   \smallskip Note that, for each  $\alpha \in \mathcal{F}$, for each $n$, if $n=\mu i[\alpha(i)\neq 0]$, then there exists $\beta$ in $\mathcal{F}_{n'}$ such that $\alpha=\overline\alpha (n+1)\ast\beta$.

  As, for each $n$, $\mathcal{F}_n\subseteq Fin_{(R_n)^+}\subseteq Fin_{\bigcup_i (R_i)^+}$, and $\bigcup_i(R_i)^+\subseteq \bigl(\bigcup_i R_i\bigr)^+=R$ and $R$ is shift-invariant,  conclude: $\forall\alpha\in\mathcal{F}[\exists n[\alpha(n)\neq 0]\rightarrow \alpha \in Fin_R]$, that is: $\mathcal{F}\subseteq Fin_{R^+}$. 
  
  \smallskip Now assume $\mathcal{F}\subseteq Fin_R$, that is: $\forall \alpha \in \mathcal{F}\exists n\forall m>n[\alpha(m)\neq 0]\rightarrow \exists i[\alpha \in Fin_{R_i}]]$.  Using Lemma \ref{L:bcpspreads},  find $p,n$ such that \\$\forall \alpha \in \mathcal{F}[\overline{\underline 0} p\sqsubset \alpha \rightarrow \forall m >n[\alpha(m)\neq 0\rightarrow \exists i[\alpha \in Fin_{R_i}]]$. 
  
  Define $q:=\max(p, n+1)$ and note:
  $\forall \alpha \in \mathcal{F}[\overline{\underline 0}q\ast\langle 1\rangle\sqsubset \alpha\rightarrow \exists i[\alpha\in \mathcal{F}_i]]$.   
  
  Using Lemma \ref{L:bcpspreads} again, find $r,i$ such that $\forall\alpha \in \mathcal{F}[\overline{\underline 0}q\ast\langle 1 \rangle\ast\overline{\underline 0}r\sqsubset \alpha\rightarrow \alpha \in \mathcal{F}_i]$.

  Find $n\ge q+r+1$ such that $n' =i$ and define $t:=n-(q+1)$.
  
  Note: $t\ge r$ and conclude: $\forall \beta \in \mathcal{F}_i[\overline{\underline 0}q \ast \langle 1\rangle\ast \overline{\underline 0}t\ast\langle 1 \rangle\ast \beta \in Fin_{R_i}]$. 
  
  As $R_i$ is shift-invariant,  conclude: $\mathcal{F}_i\subseteq Fin_{R_i}$. 
  
  Contradiction, as $\neg(\mathcal{F}_i\subseteq Fin_{R_i})$.
  
  Conclude: $\neg(\mathcal{F} \subseteq Fin_{R})$. 
  \end{proof}
  
  \begin{corollary}\label{C:hierarchy} For each $R$ in $\mathcal{E}$, $R\subseteq R^+$ and $\neg(R^+\subseteq R)$. \end{corollary}
  
  \begin{proof} Assume we find $R$ in $\mathcal{E}$ such that $R=R^+$. 
  
  Conclude, by induction on $\mathcal{E}$: for all $U$ in $\mathcal{E}$, $U\subseteq R$.
  
  Using Lemma \ref{L:eestar},  find $T$ in $\mathcal{E}^\ast$ such that $R\subseteq T$. 
  
  By Theorem \ref{T:hierarchy}, $T\subseteq T^+$ and $\neg(T^+\subseteq T)$.
  
  On the other  hand, $T^+\subseteq R\subseteq T$. 
  
  Contradiction. \end{proof}

   \begin{definition}\label{D:almost} We define binary relations $\sim_V^{\neg\neg}$ and $\sim_V^{almost}$ on $\mathcal{N}$, as follows.

  For all $\alpha, \beta$, $\alpha\sim_V^{\neg\neg}\beta \leftrightarrow \neg\neg\exists n\forall m>n[\alpha(n)=\beta(n)]\leftrightarrow \neg\neg (\alpha\sim_V\beta)$, and 
  
   $\alpha \sim^{almost}_V\beta \leftrightarrow \forall \zeta \in [\omega]^\omega\exists n[\alpha\circ \zeta(n)=\beta\circ\zeta(n)]$. \end{definition}
   
   $\alpha\sim_V^{almost}\beta$ if and only if the set $\{n\mid \alpha(n)\neq \beta(n)\}$ is \textit{almost$^\ast$-finite} in the sense used in \cite[Section 0.8.2]{veldman2005}.
   
   \smallskip
   The following  axiom is a form of Brouwer's famous \textit{Thesis on bars in $\mathcal{N}$}, see \cite{veldman2006b}. 
    \begin{axiom}[The Principle of Bar Induction]\label{A:barinduction} $\;$

  For all $B,C\subseteq \mathbb{N}$, if $\forall \alpha \exists n[\overline \alpha n \in B]$ and $B\subseteq C$ and $\forall s[s\in C\leftrightarrow \forall n[s\ast\langle n\rangle \in C]]$, then $\langle \; \rangle \in C$,  
  
  or, equivalently,
  
  for all $B,C\subseteq [\omega]^{<\omega}$, if $\forall \zeta \in [\omega]^\omega \exists n[\overline \zeta n \in B]$ and $B\subseteq C$ and \\$\forall s\in [\omega]^{<\omega}[s\in C\leftrightarrow \forall n[s\ast\langle n\rangle \in [\omega]^{<\omega}\rightarrow s\ast\langle n\rangle\in C]]$, then $\langle \; \rangle \in C$. \end{axiom}
  
   \begin{theorem}\label{T:vitalistable}$\:$
   
   \begin{enumerate}[\upshape (i)]\item $\sim_V^{\neg\neg}$ and $\sim_V^{almost}$ are equivalence relations on $\mathcal{N}$.
   \item For all $R$ in $\mathcal{E}$, $\sim_V\;\subseteq R\subseteq \;\sim_V^{\neg\neg}$.   \item For all $R$ in $\mathcal{E}$, $R\;\subseteq\; \sim^{almost}_V $. \item  $\forall\alpha\forall\beta[ \alpha\sim_V^{almost}\beta\rightarrow\exists R\in\mathcal{E}[\alpha R\; \beta]$. \item $\forall \alpha\forall \beta[\alpha\sim_V^{almost}\beta\rightarrow \alpha \sim_V^{\neg\neg}\beta]$. \end{enumerate} \end{theorem}
  
  \begin{proof} 
  
   (i) One easily proves that $\sim_V^{\neg\neg}$ is an equivalence relation. One needs the fact that, for all propositions $P,Q$, $(\neg\neg P\;\wedge\;\neg\neg Q)\rightarrow \neg\neg(P\;\wedge\; Q)$.
   
   \smallskip We prove that $\sim_V^{almost}$ is a transitive relation.
  
  Let $\alpha, \beta, \gamma$ be given such that $\alpha \sim_V^{almost}\beta$ and $\beta\sim_V^{almost} \gamma$.
  
  Let $\zeta$ in $[\omega]^\omega$ be given. Find $\eta$ in $[\omega]^\omega$ such that $\forall n[\alpha\circ\zeta\circ\eta(n)=\beta\circ\zeta\circ\eta(n)]$. Find $p$ such that $\beta\circ\zeta\circ\eta(p)=\gamma\circ\zeta\circ\eta(p)$. Define $n:=\eta(p)$ and note: $\alpha\circ\zeta(n)=\gamma\circ\zeta(n)$. 
  
  We thus see: $\forall \zeta\in [\omega]^\omega\exists n[\alpha\circ\zeta(n)=\gamma\circ\zeta(n)]$, that is: $\alpha \sim^{almost}_V\gamma$.

 \smallskip (ii) The proof is by (transfinite) induction on $\mathcal{E}$. We only prove: for all $R$ in $\mathcal{E}$, $R\subseteq \;\sim_V^{\neg\neg}$ as the statement: for all $R$ in $\mathcal{E}$, $\sim_V\;\subseteq R$ is very easy to prove. 
  
  \smallskip
  
 (I) Our starting point is the trivial observation: $\forall\alpha\forall\beta[\alpha \sim_V\beta\rightarrow \neg\neg (\alpha\sim_V \beta)]$.

  \smallskip
 (II) Now let $R$ in $\mathcal{E}$ be given such that $\forall\alpha\forall\beta[\alpha R\beta\rightarrow \neg\neg(\alpha\sim_V \beta)]$. 
  
  We have to prove: $\forall\alpha\forall\beta[\alpha R^+\beta\rightarrow \neg\neg(\alpha\sim_V \beta)]$.  
  
  We do so as follows.
  
  Let $\alpha, \beta$ be given such that $\alpha R^+\beta$. \\Find $n$ such that $\forall m>n[\alpha(m)\neq \beta(m) \rightarrow \alpha R\beta]$ and consider two special cases.
  
  \textit{Case (1):} $\exists m>n[\alpha(m)\neq \beta(m)$. Then $\alpha R\;\beta$, and, therefore: $\neg\neg(\alpha\sim_V\beta)$.
  
  \textit{Case (2):} $\neg\exists m>n[\alpha(m) \neq \beta(m)$. Then $\forall m>n[\alpha(m)=\beta(m)]$ and $\alpha\sim_V\beta$.
  
  In both cases, we find: $\neg\neg(\alpha\sim_V\beta)$.
  
  Conclude\footnote{using the scheme: if $P\rightarrow Q$ and $\neg P\rightarrow Q$, then $\neg\neg Q$.}: $\neg\neg(\alpha\sim_V\beta)$.

(III)  Now let $R_0, R_1, \ldots$ be an infinite sequence of elements of $\mathcal{E}$ such that, for all $n$, $\forall \alpha\forall\beta[\alpha R_n\beta \rightarrow \neg\neg(\alpha\sim_V\beta)]$. 

Define $R:=\bigcup_n R_n$ and note: $\forall \alpha\forall\beta[\alpha R\beta \rightarrow \neg\neg(\alpha\sim_V\beta)]$. 

 \smallskip (iii) The proof is by (transfinite) induction on $\mathcal{E}$. 
  
  \smallskip
  
 (I) Our starting point is the observation: $\forall\alpha\forall\beta[a\sim^0_V\beta\rightarrow \alpha\sim^{almost}_V \beta]$. \\We prove this as follows:
  
  Let $\alpha,\beta$ be given such that $\alpha\sim^0_V \beta$. Find $n$ such that $\forall m>n[\alpha(m)=\beta(m)]$. Note: $\forall \zeta \in [\omega]^\omega][\zeta(n+1)>n \;\wedge\;\alpha\circ\zeta(n+1)=\beta\circ\zeta(n+1)]$. \\Conclude: $\alpha\sim_V^{almost}\beta$. 
  
  \smallskip
 (II) Now let $R$ in $\mathcal{E}$ be given such that $\forall\alpha\forall\beta[\alpha R\beta\rightarrow \alpha\sim^{almost}_V \beta]$. 
  
  We have to prove: $\forall\alpha\forall\beta[aR^+\beta\rightarrow \alpha\sim^{almost}_V \beta]$.  
  
  We do so as follows.
  
  Let $\alpha, \beta$ be given such that $\alpha R^+\beta$. \\Find $n$ such that $\forall m>n[\alpha(m)\neq \beta(m) \rightarrow \alpha R\beta]$. Let $\zeta$ in $[\omega]^\omega$ be given. Consider $\zeta(n+1)$ and note $\zeta(n+1)>n$. There now are two cases. 
  
  \textit{Either} $\alpha\circ\zeta(n+1)=\beta\circ\zeta(n+1)$ \textit{or} $\alpha\circ\zeta(n+1)\neq\beta\circ\zeta(n+1)$.
  
  In the first case we are done, and in the second case we conclude $\alpha R \beta$, and, using the induction hypothesis, find $p$ such that $\alpha\circ\zeta(p)=\beta\circ\zeta(p)$. 
  
  In both cases we  conclude: $\exists q[\alpha\circ\zeta(q)=\beta\circ\zeta(q)]$.
  
  We thus see: $\forall \zeta \in [\omega]^\omega\exists q[\alpha\circ\zeta(q)=\beta\circ\zeta(q)]$, that is $\alpha\sim^{almost}_V\beta$.
  
  Clearly then: $\forall \alpha\forall\beta[[\alpha R^+\beta\rightarrow \alpha\sim^{almost}_V \beta]$.  
  
  \smallskip
(III)  Now let $R_0, R_1, \ldots$ be an infinite sequence of elements of $\mathcal{E}$ such that, for all $n$, $\forall \alpha\forall\beta[\alpha R_n\beta \rightarrow \alpha\sim_V^{almost}\beta]$. 

Define $R:=\bigcup_n R_n$ and note: $\forall \alpha\forall\beta[\alpha R\beta \rightarrow \alpha\sim_V^{almost}\beta]$.
  
  \smallskip
  (iv) Let $\alpha,\beta$ be given such that $\alpha\sim^{almost}\beta$, that is: \\$\forall \zeta \in [\omega]^\omega\exists n[\alpha\circ\zeta(n)=\beta\circ\zeta(n)]$.\\ Using Axiom \ref{A:barinduction}, we shall prove: there exists $R$ in $\mathcal{E}$ such that $\alpha R\beta$. 
  
 \smallskip Define $B:=\bigcup_k\{s\in [\omega]^{k+1}\mid \alpha\circ s(k)=\beta\circ s(k)\}$ and note: $B$ is a bar in $[\omega]^\omega$, that is: $\forall \zeta \in [\omega]^\omega\exists n[\overline \zeta n \in B]$.

  Define $C:= \bigcup_k\{ s \in [\omega]^{k}\mid \exists n<k[\alpha\circ s(n)=\beta\circ s(n)] \;\vee\;\exists R \in \mathcal{E}[\alpha R \beta]\}$.
  
  Note: $C= \bigcup_k\{ s \in [\omega]^{k}\mid \forall n<k[\alpha\circ s(n)\neq\beta\circ s(n)] \rightarrow\exists R \in \mathcal{E}[\alpha R \beta]\}$.
  
  Note: $B\subseteq C$ and: $C$ is \textit{monotone}, that is:\\$\forall s\in[\omega]^{<\omega}[s \in C \rightarrow \forall n[s\ast\langle n \rangle  \in [\omega]^{<\omega}\rightarrow s\ast\langle n \rangle \in C]]$. 
  
  \smallskip
  We still have to prove that $C$ is what one calls \textit{inductive} or \textit{hereditary}.

  Let $s$ in $[\omega]^{< \omega}$ be given such that $\forall n[s\ast\langle n \rangle  \in [\omega]^{<\omega}\rightarrow s\ast\langle n \rangle \in C]$. \\We want to prove: $s\in C$. 
  
  Find $k$ such that $s\in [\omega]^k$.
  In case $\exists n<k[\alpha\circ s(n) =\beta\circ s(n)]$, $s \in C$ and 
   we are done, so we assume: $\forall n<k[\alpha\circ s(n)\neq\beta\circ s(n)]$.
   
  Find a sequence\footnote{This application of countable choice may be reduced to Axiom \ref{ax:countable choice}. One may define $\mathcal{B}\subseteq \mathcal{N}$ and a \textit{coding} mapping $\alpha\mapsto R_\alpha$ such that $\mathcal{E}=\{R_\alpha\mid\alpha\in \mathcal{B}\}$.  }  $R_0, R_1, \ldots$ of elements of $\mathcal{E}$ such that, for each $n$, if $s\ast\langle n \rangle \in [\omega]^\omega$ and $\alpha(n)\neq \beta (n)$, then $\alpha R_n \beta$. 
  
  Define $R:=(\bigcup_i R_i)^+$ and note: $R\in \mathcal{E}$. 
  
  We claim: $\alpha R\beta$. 
  
  We establish this claim as follows. 
  
 Define $p$ such that, if $k=0$, then $p:=0$ and, if $k>0$, then $p:=s(k-1)+1$. 
   
 Assume: $\exists n\ge p[s\ast \langle \alpha(n)\neq\beta(n)]$ and find $n\ge p$ such that $\alpha(n)\neq \beta(n)$. 
 
 Note: $s\ast\langle n \rangle \in [\omega]^{k+1}$ and $\forall i<k+1[\alpha\circ(s\ast\langle n \rangle)(i)\neq  \beta\circ(s\ast\langle n \rangle)(i)]$ and $s\ast\langle n \rangle \in C$. Conclude: $\alpha R_n\beta$ and $\alpha( \bigcup_i R_i )\beta$.
 
 We thus see: $\forall n\ge p[\alpha(n)\neq \beta(n)\rightarrow\alpha( \bigcup_i R_i )\beta ]$.
 
 Conclude: $\alpha ( \bigcup_i R_i )^+ \beta$, that is: $\alpha R \beta$, and, therefore: $s\in C$. 
   
   We thus see that $C$ is inductive.
 
  \smallskip Using Axiom \ref{A:barinduction}, we conclude: $\langle\;\rangle\in C$, that is: $\exists R\in \mathcal{E}[\alpha R \beta]$.

  \smallskip (v) Let $\alpha,\beta$ be given such that $\alpha\sim^{almost}\beta$, that is: \\$\forall \zeta \in [\omega]^\omega\exists n[\alpha\circ\zeta(n)=\beta\circ\zeta(n)]$.\\ Using Axiom \ref{A:barinduction}, we prove: $\neg\neg\exists p\forall n>p[\alpha(n)=\beta(n)]$. 
  
 \smallskip Define $B:=\bigcup_{k}\{s\in [\omega]^{k+1}\mid \alpha\circ s(k)=\beta\circ s(k)\}$ and note: $B$ is a bar in $[\omega]^\omega$, that is: $\forall \zeta \in [\omega]^\omega\exists n[\overline \zeta n \in B]$. 
  Define \\$C:=\bigcup_k \{ s \in [\omega]^k\mid \exists n<k[\alpha\circ s(n)=\beta\circ s(n)]\;\vee\;\neg\neg\exists p\forall n>p[\alpha(n)=\beta(n)]\}$.
  Note: $C=\bigcup_k \{ s \in [\omega]^k\mid \forall n<k[\alpha\circ s(n)\neq\beta\circ s(n)]\rightarrow\neg\neg\exists p\forall n>p[\alpha(n)=\beta(n)]\}$.
  
  Note: $B\subseteq C$ and $C$ is monotone, that is: \\$\forall s\in[\omega]^{<\omega}[s \in C \rightarrow \forall n[s\ast\langle n \rangle  \in [\omega]^{<\omega}\rightarrow s\ast\langle n \rangle \in C]]$. 
  
  We still have to prove that $C$ is inductive.
  
   Let $s$ in $[\omega]^{< \omega}$ be given such that $\forall n[s\ast\langle n \rangle  \in [\omega]^{<\omega}\rightarrow s\ast\langle n \rangle \in C]]$. \\We want to prove: $s\in C$. 
  
  Find $k$ such that $s\in [\omega]^k$. In case $\exists n<k[\alpha\circ s(n)=\beta\circ s(n)]$, $s\in C$,  and we are done, so we assume 
  $\forall n<k[\alpha\circ s(n)\neq\beta\circ s(n)]$.
  
  Define $q$ such that $q:=0$ if $k=0$ and $q:=s(k-1)$ if $k>0$.
  
   Consider two special cases:
  
  \textit{Case (1):} $\exists n>q[\alpha (n) \neq\beta(n)]$.\\ Find such $n$, note: $s\ast\langle n \rangle \in [\omega]^\omega$ and $\forall i<k+1[\alpha\circ (s\ast\langle n \rangle)(i)\neq\beta\circ (s\ast\langle n \rangle)(i)]$ and $s\ast\langle n \rangle \in C$, and conclude:  $\neg\neg\exists p\forall n>p[\alpha(n)=\beta(n)]$. 
  
  \textit{Case (2):} $\neg\exists n>q[\alpha (n) \neq\beta(n)]$, and, therefore,  $\forall n>q[\alpha(n)=\beta(n)]$. 
  
  In both cases, we find: $\neg\neg\exists p\forall n>p[\alpha(n)=\beta(n)]$. 
   
   Conclude\footnote{Using the scheme: If $P\rightarrow Q$ and $\neg P \rightarrow Q$, then $\neg\neg Q$.}:
   $\neg\neg\exists p\forall n>p[\alpha(n)=\beta(n)]$, and: $s\in C$.
   
   We thus see that $C$ is inductive.
   
   \smallskip Using Axiom \ref{A:barinduction}, we conclude: $\langle\;\rangle \in C$, and, therefore,
   
    $\neg\neg\exists p\forall n>p[\alpha(n)=\beta(n)]$, that is: $\neg\neg(\alpha\sim_V \beta)$. 
  \end{proof}
  
  \begin{corollary}\label{C:unstableextensions} \begin{enumerate}[\upshape (i)] \item $(\mathcal{N}, \sim_V^{\neg\neg})\models \mathsf{\forall x\forall y[\neg\neg(x=y)\rightarrow x=y]}$. \item For each $R$ in $\mathcal{E}$, $(\mathcal{N}, R)\models \mathsf{\forall x\neg\forall y[\neg\neg(x=y)\rightarrow x=y]}$.
  \end{enumerate}\end{corollary}
  
  \begin{proof} (i) Obvious, as, for any proposition $P$, $\neg\neg\neg\neg P\leftrightarrow \neg\neg P$.
  
  \smallskip (ii) Assume $R\in \mathcal{E}$. 
  
  We first prove: $(\mathcal{N}, R)\models \mathsf{\neg\forall x\forall y[\neg\neg(x=y)\rightarrow x=y]}$.  
  
  Assume  $\forall \alpha\forall \beta[\neg\neg(\alpha R \beta)\rightarrow \alpha R \beta]$. 
  
  Note: $\forall\alpha\forall \beta[\alpha \sim_V\beta\rightarrow\alpha R\beta]$ and, therefore: $\forall\alpha\forall \beta[\neg\neg(\alpha \sim_V\beta)\rightarrow\neg\neg(\alpha R\beta)]$.

  Conclude: $\sim_V^{\neg\neg}\; \subseteq R$.  
  
  By Theorem \ref{T:vitalistable}(ii), $R^+\subseteq \sim_V^{\neg\neg}$, so $R^+\subseteq R$. This contradicts Corollary \ref{C:hierarchy}.
  
  \smallskip
  The stronger statement announced in the Theorem may be proven in a similar way.
  Inspection of he proof of Theorem \ref{T:vitalistable} enables one to conclude: 
  
  $(\mathcal{N}, R)\models \mathsf{\neg \forall y[\neg\neg(x=y)\rightarrow x=y]}[\underline 0]$.  
 One easily generalizes this conclusion to:\\ for each $\alpha$, $(\mathcal{N}, R)\models \mathsf{\neg \forall y[\neg\neg(x=y)\rightarrow x=y]}[\alpha]$.  
 
 Conclude: $(\mathcal{N}, R)\models \mathsf{\forall x\neg\forall y[\neg\neg(x=y)\rightarrow x=y]}$.
  \end{proof}
  
  Markov's Principle has been mentioned in Section \ref{S:spreads}. Markov's Principle is not accepted in intuitionistic mathematics, but the following observation still is of interest. 
  
  \begin{corollary} The following are equivalent. \begin{enumerate}[\upshape (i)]\item Markov's Principle: $\forall\alpha[\neg\neg\exists n[\alpha(n)=0]\rightarrow \exists n[\alpha(n)=0]]$. \item $\sim_V^{\neg\neg}\;\subseteq \;\sim_V^{almost}$. \item $\sim_V^{almost}$ is stable. \end{enumerate}\end{corollary} \begin{proof} (i) $\Rightarrow$ (ii).  Assume $\neg\neg (\alpha\sim_V\beta)$, that is $\neg\neg\exists n\forall m>n[\alpha(m)=\beta(m)]$. 
  
  Let $\zeta\in [\omega]^\omega$ be given. 
  
  Assume: $\neg \exists n[\alpha\circ\zeta(n)=\beta\circ\zeta(n)]$. 
  
  Then $\forall n[\zeta(n+1)>n\;\wedge\;\alpha\circ\zeta(n)\neq \beta\circ\zeta(n)]$, so $\forall n\exists m>n[\alpha(m)\neq\beta(m)]$.  Contradiction. 
  
  Conclude: $\neg\neg\exists n[\alpha\circ \zeta(n)=\beta\circ\zeta(n)]$ and,  by Markov's Principle, \\$\exists n[\alpha\circ\zeta(n)=\beta\circ\zeta(n)]$. 
  
  We thus see $\forall\zeta\in [\omega]^\omega \exists n[\alpha\circ\zeta(n)=\beta\circ\zeta(n)]$, that is: $\alpha\sim_V^{almost} \beta$. 
  
  \smallskip (ii) $\Rightarrow$ (iii). By Theorem \ref{T:vitalistable}(v), $\sim_V^{almost}\;\subseteq\; \sim_V^{\neg\neg}$. Therefore: $(\sim_V^{almost})^{\neg\neg}\;\subseteq\; \sim_V^{\neg\neg}$. 
  
  Using (ii), we conclude: $(\sim_V^{almost})^{\neg\neg}\;\subseteq \;\sim_V^{almost}$, that is: $\sim_V^{almost}$ is stable. 
  
  \smallskip (iii) $\Rightarrow$ (i). Let $\alpha$ be given such that $\neg\neg\exists n[\alpha(n)\neq 0]$. 
  
  Define $\beta$ such that $\forall m[\beta(m)=0\leftrightarrow \exists n\le m[\alpha(n)=0]]$.
  
   Note: $\neg\neg(\beta\sim_V\;\underline 0)$ and, therefore: $\neg\neg (\beta\sim_V^{almost}\underline 0)$. 
  
  Conclude, using (iii), $\beta\sim_V^{almost} \underline 0$.
  
  Define $\zeta$ such that $\forall n[\zeta(n)=n]$. 
  
  Find $m$ such that $\beta \circ\zeta(m)=\beta(m)=0$ and, therefore, $\exists n\le m[\alpha(n) =0]$.
  
  We thus see: $\forall\alpha[\neg\neg\exists n[\alpha(n)=0]\rightarrow \exists n[\alpha(n)=0]]$, that is: Markov's Principle. \end{proof}
\section{Equality and equivalence}\label{S: equequiv} We did not succeed in finding a sentence $\psi$ such that $(\mathcal{N}, \sim_V)\models\psi$ and \\$(\mathcal{N}, \sim^\omega_V)\models \neg \psi$.
 We now want to compare the structures  $(\mathcal{N}, =, \sim_V)$ and   \\$(\mathcal{N}, =, \sim_V^\omega)$. We need a first order language with two binary relation symbols: $=$ and $\sim$. The  symbol $=$ will denote the equality relation and the symbol $\sim$ will denote, in the first structure, the relation $\sim_V$ and, in the second structure, the relation $\sim_V^\omega$. The reader hopefully will not be confused by the fact that, in the earlier sections, where we used the first order language with  a single binary relation symbol, $=$,  the symbol $=$ denoted the relations $\sim_V$ and $\sim_V^\omega$.
 
 The next Theorem makes us see that equality is decidable on each equivalence class of $\sim_V$ whereas, on each equivalence class of $\sim_V^\omega$, it is not decidable.
 
 \begin{theorem}\label{T:vitaliomegavitali}$\;$
 
 \begin{enumerate}[\upshape (i)] \item $(\mathcal{N}, =, \sim_V)\models \mathsf{\forall x\forall y[x\sim y\rightarrow (x=y\;\vee\;\neg(x=y))]}$. \item $(\mathcal{N}, =, \sim_V^\omega)\models \mathsf{\forall x\neg\forall y[x\sim y\rightarrow (x=y\;\vee\;\neg(x=y)]}$. \end{enumerate} \end{theorem}
 
 \begin{proof} (i) Let $\gamma, \alpha$ be given such that $\gamma\sim_V\alpha$. \\Find $n$ such that $\forall m>n[\gamma(m)=\alpha(m)]$ and distinguish two cases. \\\textit{Either} $\overline \gamma (m+1)=\overline \alpha(m+1)$ and  $\gamma=\alpha$, \textit{or}  $\overline \gamma (m+1)\neq\overline \alpha(m+1)$ and $\neg(\gamma =\alpha)$. 
 
 Conclude: $\forall \gamma\forall \alpha[\gamma\sim_V\alpha\rightarrow (\gamma =\alpha \;\vee\;\neg(\gamma=\alpha)]$. 
 
 \smallskip (ii) Let $\gamma$ be given.
 
  Consider $\mathcal{F}_1^\gamma := \{\alpha\mid \forall m\forall n[(\alpha(m)\neq\gamma(m)\;\wedge\;\alpha(n)\neq\gamma(n)\rightarrow m=n]\}$.
 
 Note: $\mathcal{F}_1^\gamma$ is a spread. 
 
 Also:  $\forall \alpha \in \mathcal{F}_1^\gamma[\gamma\sim_V^1\alpha]$ \footnote{See the proof of Theorem \ref{T:vitalivariations}(iii)}  and, therefore, $\forall \alpha \in \mathcal{F}_1^\gamma[\gamma\sim_V^\omega \alpha]$. 
 
 Assume  $\forall \alpha\in \mathcal{F}_1^\gamma[\gamma=\alpha\;\vee\;\neg(\gamma=\alpha)]$. Applying Lemma \ref{L:bcpdisj}, find $p$ such that \textit{either} $\forall \alpha\in\mathcal{F}_1^\gamma[ \overline \gamma p\sqsubset \alpha\rightarrow \gamma =\alpha]$ \textit{or} $\forall \alpha[ \overline \gamma p\sqsubset \alpha\rightarrow \neg(\gamma =\alpha)]$, and note that both alternatives are false. 
 
 Conclude: $\forall \gamma\neg\forall \alpha[\gamma\sim_V^\omega \alpha\;\vee\neg(\gamma=\alpha)]$. \end{proof}
 
 \begin{lemma}\label{L:almostperhapsive} $(\sim_V^{\neg\neg})^+\;\subseteq\;\;\sim_V^{\neg\neg}$ and  $(\sim_V^{almost})^+\;\subseteq\;\;\sim_V^{almost}$.\footnote{Following the terminology in \cite{veldman1995}, a binary relation $R$ on $\mathcal{N}$ should be called \textit{perhapsive} if $R^+\subseteq R$.}
 \end{lemma}
 
 \begin{proof} Assume $\alpha (\sim_V^{\neg\neg})^+ \beta$.
 
  Find $n$ such that $\forall m>n[\alpha(m)\neq \beta(m)\rightarrow \alpha\sim_V^{\neg\neg}\beta]$. 
  
  Note: if $\exists m>n[\alpha(m)\neq \beta(m)]$, then $\alpha\sim_V^{\neg\neg} \beta$, and if $\neg\exists m>n[\alpha(m)\neq \beta(m)]$, then $\forall m>n[\alpha(m)=\beta(m)]$ and $\alpha\sim_V\beta$ and also $\alpha\sim_V^{\neg\neg}\beta$.

Conclude: $\neg\neg(\alpha\sim_V^{\neg\neg}\beta)$, and, therefore, $\alpha\sim_V^{\neg\neg}\beta$.

\medskip Assume $\alpha (\sim_V^{almost})^+ \beta$.
 
  Find $n$ such that $\forall m>n[\alpha(m)\neq \beta(m)\rightarrow \alpha\sim_V^{almost}\beta]$. 
  
  Let $\zeta$ in $[\omega]^\omega$ be given. Note: $\zeta(n+1)>n$. 
  
  \textit{Either}: $\alpha\circ\zeta(n+1)=\beta\circ\zeta(n+1)$ \textit{or}: $\alpha \sim_V^{almost}\beta$ and $\exists p[\alpha\circ\zeta(p)=\beta\circ\zeta(p)]$.
  
  We thus see: $\forall \zeta\in[\omega]^\omega\exists n[\alpha\circ\zeta(n)=\beta\circ\zeta(n)]$, that is: $\alpha\sim_V^{almost}\beta$. \end{proof}
  
  \begin{lemma}\label{L:lastformula} For every shift-invariant binary relation $R$ on $\mathcal{N}$, 
  
  $R^+\subseteq R$ if and only if $(\mathcal{N},R)\models \mathsf{\forall x \forall y[}\bigl(AP(\mathsf{x,y)\rightarrow x\sim y\bigr)\rightarrow x\sim y]}$. \end{lemma}
  
  \begin{proof} First assume $R^+\subseteq  R$.
  
  Assume $\alpha\;\#\;\beta\rightarrow \alpha R \beta$.
  
  Then: $\forall m>0[\alpha(m)\neq\beta(m)\rightarrow \alpha R\beta]$, so: $\alpha R^+\beta$, and, therefore: $\alpha R\beta$. 
  
  We thus see: $(\mathcal{N},R)\models \mathsf{\forall x \forall y[}\bigl(AP(\mathsf{x,y)\rightarrow x\sim y\bigr)\rightarrow x\sim y]}$. 
  
  \medskip Now assume $(\mathcal{N},R)\models \mathsf{\forall x \forall y[}\bigl(AP(\mathsf{x,y)\rightarrow x\sim y\bigr)\rightarrow x\sim y]}$. 
  
  Assume $\alpha R^+\beta$. Find $n$ such that $\forall m>n[\alpha(m)\neq \beta(m)\rightarrow \alpha R \beta]$. 
  
  Define $\gamma, \delta$ such that $\forall m[\gamma (m)=\alpha(n+1+m)\;\wedge\;\delta(m) =\beta(n+1+m)]$.
  
  Note: $\gamma \;\#\;\delta \rightarrow \alpha R \beta$, and, as
  $R$ is shift-invariant, also: $\gamma\;\#\;\delta \rightarrow \gamma R \delta$, and, therefore: $\gamma R \delta$, and also: $\alpha R\beta$. 
  
  We thus see: $R^+\subseteq R$. \end{proof}
  
  \begin{corollary}\label{C:lastform} \begin{enumerate}[\upshape (i)]\item $(\mathcal{N}, \sim_V^{\neg\neg})\models \mathsf{\forall x \forall y[}\bigl(AP(\mathsf{x,y)\rightarrow x\sim y\bigr)\rightarrow x\sim y]}$.
  
  \item $(\mathcal{N}, \sim_V^{almost})\models \mathsf{\forall x \forall y[}\bigl(AP(\mathsf{x,y)\rightarrow x\sim y\bigr)\rightarrow x\sim y]}$.
  
  \item For each $R$ in $\mathcal{E}$, $(\mathcal{N},R)\models \mathsf{\neg\forall x \forall y[}\bigl(AP(\mathsf{x,y)\rightarrow x\sim y\bigr)\rightarrow x\sim y]}$. \end{enumerate}\end{corollary}
  
  \begin{proof}
  Use Lemmas \ref{L:almostperhapsive} and \ref{L:lastformula} and Corollary \ref{C:hierarchy}. \end{proof}
  
\section{Notations and conventions}\label{S:notations}

We use $m,n, \ldots$ as variables over the set $\omega=\mathbb{N}$ of the natural numbers.

\smallskip For every $P\subseteq \mathbb{N}$ such that $\forall n[P(n)\;\vee\;\neg P(n)]$, for all $m$,

$m=\mu n[P(n)]\leftrightarrow \bigl(P(m)\;\wedge\;\forall n<m[\neg P(n)]\bigr)$.

\smallskip
$(m,n)\mapsto J(m,n)$ is a one-to-one surjective mapping from $\omega\times\omega$ onto $\omega$. 

$K,L:\omega\times \omega$ are its inverse functions, so $\forall n[J\bigl(K(n), L(n)\bigr)=n]$.

For each $n$, $n':=K(n)$ and $n'':=L(n)$.

\smallskip
 $(n_0, n_1, \ldots, n_{k-1})\mapsto \langle n_0, n_1, \ldots, n_{k-1}\rangle$ is a one-to-one surjective mapping from the set of finite sequences of natural numbers to the set of the natural numbers. 

$\langle n_0, n_1, \ldots, n_{k-1}\rangle$ is the \textit{code} of the finite sequence $(n_0, n_1, \ldots, n_{k-1})$.  

 $s\mapsto length(s)$ is is the function that, for each $s$, gives the length of the finite sequence coded by $s$.

 $s,n\mapsto s(n)$ is the function that, for all $s,n$, gives the value of the finite sequence coded by $s$ at $n$. If $n\ge length(s)$, then $s(n)=0$.

For all $s,k$, if $length(s)=k$, then $s=\langle s(0), s(1), \ldots s(k-1)\rangle$.

$0=\langle\;\rangle$ codes the empty sequence of natural numbers,\\ the unique finite sequence $s$ such that $length(s)=0$.

$\omega^k:=\{s\mid length(s)=k\}$.

\smallskip $[\omega]^k:=\{s\in \omega^k\mid\forall i[i+1<k\rightarrow s(i)<s(i+1)]\}$. 

\smallskip $[\omega]^{<\omega}:=\bigcup_k[\omega]^k$.

\smallskip For all $s,k,t,l$, if $s\in \omega^k $ and $t\in \omega^l$, then $s\ast t$ is the element $u$ of $\omega^{k+l}$ such that $\forall i<k[u(i)=s(i)]$ and  $\forall j<l[u(k+j)=t(j)]$.

\smallskip $s\sqsubseteq t\leftrightarrow \exists u[s\ast u =t]$.

\smallskip $s\sqsubset t\leftrightarrow (s\sqsubseteq t \;\wedge\;s\neq t)$. 

\medskip
We use $\alpha, \beta, \ldots$ as variables over \textit{Baire space}, the set $\omega^\omega:=\mathcal{N}$ of functions from $\mathbb{N}$ to $\mathbb{N}$.

$(\alpha, n)\mapsto \alpha(n)$ is the function that associates to all $\alpha, n$, the value of $\alpha$ at $n$. 

For all $\alpha, \beta$, $\alpha\circ\beta$ is the element $\gamma$ of $\mathcal{N}$ such that $\forall n[\gamma(n)=\alpha\bigl(\beta(n)\bigr)]$. 

\smallskip $2^\omega:=\mathcal{C}:=\{\alpha\mid\forall n[\alpha(n)<2]\}$ is \textit{Cantor space}. 

For all $\alpha$, for all $k$, for all $s$ in $\omega^k$, $\alpha\circ s$ is the element $t$ of $\omega^k$ satisfying\\ $\forall n<k[t(k)=\alpha\bigl(s(k)\bigr)]$. 

\medskip
For each $s, k$, if $s \in \omega^k$, then, for each $\alpha$, $s\ast\alpha$ is the element $\beta$ of $\mathcal{N}$ such that $\forall i <k[\beta(i)=s(i)]$ and $\forall i[\beta(k+i)=\alpha(i)]$.

\smallskip
For each $s$, for each $\mathcal{X}\subseteq \mathcal{N}$, $s\ast\mathcal{X}:=\{s\ast\alpha\mid\alpha \in \mathcal{X}\}$. 

\smallskip For each $\alpha$, for each $n$, $\alpha^n$ is the element of $\mathcal{N}$ satisfying $\forall m[\alpha^n(m)=\alpha\bigl(J(n,m)\bigr)]$. 

\smallskip For each $m$, $\underline m \in \mathcal{N}$ is the element of $\mathcal{N}$ satisfying $\forall n[\underline m(n)=m]$.

\smallskip $S$ is the element of $\mathcal{N}$ satisfying $\forall n[S(n)=n+1]$.

\smallskip $\forall n[\alpha'(n)=\bigl(\alpha(n)\bigr)'\;\wedge\;\alpha''(n)=\bigl(\alpha(n)\bigr)'']$. 

\smallskip
$ \overline \alpha n :=\langle \alpha(0),\alpha(1), \ldots \alpha(n-1)\rangle$. 

\smallskip
$s\sqsubset\alpha\leftrightarrow \exists n[\overline \alpha n =s]$. 

\smallskip $\alpha\perp\beta \leftrightarrow \alpha\;\#\;\beta\leftrightarrow \exists n[\alpha(n)\neq\beta(n)]$. 

\smallskip $[\omega]^\omega:=\{\zeta\in\mathcal{N}\mid \forall i[\zeta(i)<\zeta(i+1)]\}$.

\smallskip $\mathbb{Q}$, the set of the rationals, may be defined as a subset of $\omega$, with accompanying relations $=_\mathbb{Q}$, $<_\mathbb{Q}$, $\le_\mathbb{Q}$ and operations $+_\mathbb{Q}, -_\mathbb{Q}, \cdot_\mathbb{Q}$.

\smallskip $\mathcal{R}:=\{\alpha\mid \forall n[\alpha' (n)\in \mathbb{Q}\;\wedge\; \alpha''(n)\in \mathbb{Q}]\;\wedge\\\forall n[\alpha'(n)\le_\mathbb{Q}\alpha' (n+1)\le_\mathbb{Q} \alpha''(n+1)\le_\mathbb{Q}\alpha''(n)]\;\wedge\;\forall m\exists n[\alpha''(n)-_\mathbb{Q} \alpha'(n)<_\mathbb{Q} \frac{1}{2^m}]\}$.

\smallskip For all $\alpha, \beta$ in $\mathcal{R}$,

$\alpha<_\mathcal{R}\beta \leftrightarrow \exists n[\alpha''(n)<_\mathbb{Q}\beta'(n)]$ and $\alpha=_\mathcal{R}\beta \leftrightarrow \bigl(\neg (\alpha<_\mathcal{R} \beta)\;\wedge\;\neg(\beta <_\mathcal{R}\alpha)\bigr)$.

\smallskip Operations $+_\mathcal{R}, -_\mathcal{R}$ are defined straightforwardly.

\end{document}